\documentclass[10pt]{amsart}
\usepackage{amssymb,amsbsy,amsmath,amsfonts,times, amscd}
\usepackage{latexsym,euscript,exscale}
\usepackage{graphicx,color} \usepackage{amsthm}
\usepackage{fancyhdr} \usepackage{url}
\usepackage{epigraph}\usepackage{lmodern}
\usepackage{mathrsfs}\usepackage{cancel}
\usepackage[toc,page]{appendix}\usepackage[T1]{fontenc}
\usepackage{mathtools}\usepackage[all,cmtip]{xy}
\usepackage{enumerate}\usepackage{setspace}
\usepackage{helvet}
\usepackage{float}
\usepackage{hyperref}
\usepackage{tikz} 
\numberwithin{equation}{section}
\newtheorem{theorem}{Theorem}
\newtheorem{thm}[theorem]{Theorem}

\newtheorem{cor}[theorem]{Corollary}
\newtheorem{lemma}[theorem]{Lemma}
\newtheorem{prop}[theorem]{Proposition}
\theoremstyle{definition}
\newtheorem{defi}[theorem]{Definition}
\newtheorem{problem}[theorem]{Problem}
\theoremstyle{remark}
\newtheorem{remark}[theorem]{Remark}

\newtheorem{example}[theorem]{Example}
\newtheorem{claim}{Claim}
\newtheorem*{acknowledgments}{Acknowledgments}
\numberwithin{theorem}{section}

\newcounter{my_enumerate_counter1}
\newcommand{\pushcounterx}{\setcounter{my_enumerate_counter1}{\value{enumi}}}

\newcounter{my_enumerate_counter}
\newcommand{\pushcounter}{\setcounter{my_enumerate_counter}{\value{enumi}}}
\newcommand{\popcounter}{\setcounter{enumi}{\value{my_enumerate_counter}}}

\newcommand{\twom}{\{-1,1\}^\mu}
\newcommand{\e}{\eps}
\newcommand{\cstar}{$\mathrm{C}^*$}
\newcommand{\cst}{\mathrm{C}^*}
\newcommand{\cstu}{\mathrm{C}^*_u}
\DeclareMathOperator{\Ad}{Ad}
\DeclareMathOperator{\supp}{supp}
\DeclareMathOperator{\Id}{Id}
\DeclareMathOperator{\cov}{cov}
\newcommand{\cZ}{\mathcal{Z}}

\newcommand{\cY}{\mathcal Y} 

\newcommand{\sfP}{\mathsf P}
\newcommand{\sfR}{\mathsf R}

\newcommand {\D}{\mathbb D}

\newcommand {\C}{\mathbb C}

\newcommand{\eps}{\varepsilon}

\newcommand{\cB}{\mathcal{B}}
\newcommand{\cF}{\mathcal{F}}

\newcommand{\cP}{\mathcal{P}}

\newcommand{\cN}{\mathcal{N}}
\newcommand{\cE}{\mathcal{E}}
\newcommand{\cU}{\mathcal{U}}

\newcommand{\fX}{\mathfrak{X}}
\newcommand{\fE}{\mathfrak{E}}

\newcommand{\N}{\mathbb{N}}
\newcommand{\R}{\mathbb{R}}
\newcommand{\Z}{\mathbb{Z}}

\newcommand{\tn}{{\vert\kern-0.25ex\vert\kern-0.25ex\vert}}
\newcommand{\Tn}{{\big\vert\kern-0.25ex\big\vert\kern-0.25ex\big\vert}}    
\newcommand{\TN}{{\Big\vert\kern-0.25ex\Big\vert\kern-0.25ex\Big\vert}} 
\begin{document}

\title[On the rigidity of uniform Roe algebras ]{On the rigidity of uniform Roe algebras over uniformly locally finite coarse spaces}
 
\author{B. M. Braga and  I. Farah }
\address{Department of Mathematics and Statistics\\
 York University\\
4700 Keele St.\\ Toronto, Ontario, M3J IP3\\
Canada}\email{demendoncabraga@gmail.com}
\address{Department of Mathematics and Statistics\\
 York University\\
4700 Keele St.\\ Toronto, Ontario, M3J IP3\\
Canada}
\email{ifarah@mathstat.yorku.ca}
\thanks{BMB was supported by York Science Research Fellowship. Both authors were partially supported 
  by IF's NSERC grant}
\date{\today}
\maketitle

\begin{abstract}
Given a coarse space $(X,\mathcal{E})$, one can define a \cstar-algebra $\cstu(X)$ called the uniform Roe algebra of $(X,\mathcal{E})$. It has been proved by J. \v{S}pakula and R. Willett that if the uniform Roe algebras of two uniformly locally finite metric spaces with property A are isomorphic, then the metric spaces are coarsely equivalent to each other. In this paper, we look at the problem of generalizing this result for general coarse spaces and on weakening the hypothesis of the spaces having property A. 
\end{abstract}

\setcounter{tocdepth}{1}
\tableofcontents
\section{Introduction}
The  concept of coarse spaces generalizes the idea  of metric spaces  and gives us the appropriate framework to study    large-scale geometry. In a nutshell,  a coarse space  consists of a pair  $(X,\cE)$, where $X$ is a set and $\cE$ is a family of subsets of $X\times X$ which measures `boundedness' in $X$  (we refer the reader to Section \ref{SectionBackground}  for precise definitions of the terminology used in this introduction). A (connected) coarse space $(X,\cE)$ happens to be metrizable exactly when its coarse structure is generated by $\aleph_0$  subsets \cite[Theorem 2.55]{RoeBook}. 

Given a uniformly locally finite coarse space $(X,\cE)$, one can define a  $^*$-subalgebra  $\cstu[X]$ of $\cB(\ell_2(X)$ and its norm-closure, a \cstar-algebra $\cstu(X)$, called the \emph{algebraic uniform Roe algebra of $(X,\cE)$} and \emph{uniform Roe algebra of $(X,\cE)$}, respectively. These algebras  are named after J. Roe, who introduced a version of them  in his study of   index theory of elliptical operators
on noncompact manifolds \cite{Roe1988,Roe1993}. An important  motivation to study  this \cstar -algebra comes from its intrinsic relation with the coarse Baum-Connes conjecture and, as a consequence, to the Novikov conjecture (see \cite{Yu2000}).\footnote{Recently,  uniform
Roe algebras   and its K-theory have also been
used  in mathematical physics to study the classification of
topological phases and the topology of quantum systems \cite{Kubota2017,EwertMeyer2019}.} More to the point, the rigidity problems considered in this paper are directly concerned with the Baum--Connes conjecture (see the discussion in  \cite[Section~1.2]{SpakulaWillett2013}). 
Many coarse properties of $(X,\cE)$ reflect on \cstar-algebraic properties of $\cstu(X)$ and vice 
versa (see \cite{LiWillett2017}, \cite{Sako2013}, and \cite{WinterZacharias2010}). However, 
the rigidity question---whether the uniform Roe algebra completely determines the coarse 
structure of the coarse space---remains open.

An  isomorphism $\Phi\colon \cstu(X)\to \cstu(Y)$ is \emph{spatially implemented} 
if there exists  a unitary $U\colon \ell_2(X)\to \ell_2(Y)$ such that  $\Phi=\Ad U$ (where $\Ad U (T)=UTU^*$). 
In this case we say that $\cst(X)$ and $\cst(Y)$ are \emph{spatially isomorphic}. 
If in addition there exists $\e>0$ such that 
\[
\inf_{x\in X}\sup_{y\in Y}|\langle U\delta_x,\delta_y\rangle|>0\ \ \text{and}\ \ \inf_{y\in Y}\sup_{x\in X}|\langle U^*\delta_y,\delta_x\rangle|>0
\]
then we say that $\Phi$ is a  \emph{rigid    isomorphism} between $\cstu(X)$ and $\cstu(Y)$ 
and  that $\cstu(X)$ and $\cstu(Y)$ are \emph{rigidly    isomorphic}. 
Analogous terminology applies to the case when $\Phi$ is an isomorphism 
between algebraic Roe algebras $\cstu[X]$ and $\cstu[Y]$. 

We study the following relations between  coarse spaces $(X,\cE)$ and $(Y,\cF)$. 
\begin{enumerate}[(I)]
\item \label{I.I} $(X,\cE)$ and $(Y,\cF)$ are  coarsely equivalent.
\item \label{I.II} $(X,\cE)$ and $(Y,\cF)$ are bijectively coarsely equivalent.
\item \label{I.IsoAlg} $\cstu[X]$ and $\cstu[Y]$ are    isomorphic.
\item \label{I.IsoDiag} There is an isomorphism  $\cstu(X)\to\cstu(Y)$  taking $\ell_\infty(X)$ 
to $\ell_\infty(Y)$. 
\item \label{I.IsoRig} $\cstu(X)$ and $\cstu(Y)$ are rigidly    isomorphic.
\item \label{I.IsoSpa} $\cstu(X)$ and $\cstu(Y)$ are spatially    isomorphic.
\item \label{I.Iso} $\cstu(X)$ and $\cstu(Y)$ are    isomorphic. 
\pushcounterx
\end{enumerate}
 Notice that, although  the implication \eqref{I.II}  $\Rightarrow$  \eqref{I.I} is trivial, \eqref{I.I}  in general does not imply any of the other properties (for example if $X$ and $Y$ are connected and 
finite coarse spaces of different cardinalities, or if $X$ and $Y$ are $\R$ and $\Z$ with their standard metrizable coarse structures). This paper revolves around the following question.

\begin{problem}
Let $(X,\cE)$ and $(Y,\cF)$ be uniformly locally finite coarse spaces. Do all properties above imply \eqref{I.I}? Moreover, do we have that \[\eqref{I.II}\Leftrightarrow \eqref{I.IsoAlg} \Leftrightarrow\eqref{I.IsoDiag} \Leftrightarrow\eqref{I.IsoRig}\Leftrightarrow\eqref{I.IsoSpa}\Leftrightarrow\eqref{I.Iso}?\]
\end{problem}

The implications \eqref{I.IsoRig} $\Rightarrow$ \eqref{I.IsoSpa} $\Rightarrow$ \eqref{I.Iso} and \eqref{I.IsoAlg} $\Rightarrow$  \eqref{I.Iso} are trivial  and 
the implications  \eqref{I.II}  $\Rightarrow$   \eqref{I.IsoAlg}  and \eqref{I.II}  $\Rightarrow$   \eqref{I.IsoDiag}  are quite straightforward. We list below what is known regarding the remaining implications. 
 
\begin{enumerate} 
\item \label{SpatiallyImplemented} 
Properties \eqref{I.IsoSpa} and \eqref{I.Iso} are  equivalent. 
More precisely, every isomorphism as in \eqref{I.IsoAlg} or \eqref{I.Iso} is spatially implemented  (see   \cite{SpakulaWillett2013}, Lemma 3.1, or    Lemma \ref{LemmaSpakulaWillett} below). 
\item For  uniformly locally finite metric  spaces with property A,  
\eqref{I.Iso} implies \eqref{I.I} (\cite{SpakulaWillett2013}, Theorem 4.1). 
 \item   Properties \eqref{I.Iso} and  \eqref{I.IsoRig} are equivalent for  uniformly locally finite metric spaces
 with property A.  Moreover, for  uniformly locally finite metric spaces
 with property A, 
  \emph{every} isomorphism between their uniform Roe algebras is a rigid    isomorphism (\cite{SpakulaWillett2013}, Lemma 4.6). 
  
\item Properties \eqref{I.II} and \eqref{I.IsoDiag} are equivalent  for   uniformly locally finite  metric spaces (\cite{WhiteWillett2017}, Corollary 6.13).
\item  Properties \eqref{I.II},  \eqref{I.IsoAlg}, \eqref{I.IsoDiag},  \eqref{I.IsoRig}, \eqref{I.IsoSpa}, and \eqref{I.Iso} are all equivalent in the following situations. 
\begin{enumerate}
\item for  uniformly locally finite metric spaces with property A  (\cite{WhiteWillett2017}, Corollary 6.13), 
\item for  countable locally finite groups with proper left-invariant metrics 
(\cite[Theorem~1.1]{LiLiao2017}; the standard Cayley graph metric of a finitely generated group is proper and left-invariant). 
\end{enumerate}
\item Property \eqref{I.I} implies that $\cstu(X)$ and $\cstu(Y)$ are Morita 
equivalent (this was proved for metric spaces in  \cite[Theorem~4]{brodzki2007property}, but the proof clearly translates to   coarse spaces; see also   \cite[Proposition 3.10]{ChungLi2018}). 
  \pushcounter
\end{enumerate}

Recall, an operator $T\in \cstu(X)$ is called a \emph{ghost} 
 if $\langle T\delta_x,\delta_{x'}\rangle\to 0$, as $(x,x')\to \infty$ on $X\times X$ (see Definition \ref{DefiGhost} for details). 
 A uniformly locally  finite metric space $(X,d)$ has property A if and only if all ghost operators in $\cstu(X)$ are compact (see Proposition~11.43 of \cite{RoeBook} and Theorem 1.3 of \cite{RoeWillett2014}).\footnote{We take the advantage
of this fact and use it as an excuse not to give the definition of property A; let's just say that A appears to stand
 for `amenability'.}
  Since  there exist uniformly locally finite metric spaces with non-compact ghost operators in which 
  all ghost projections are compact (see \cite{RoeWillett2014}, Theorem 1.4), the property of a uniformly locally finite metric space having only compact ghost projections is strictly weaker than property~A. Also, it was proved in \cite{Yu2000}, Theorem 2.7, that having property~A implies coarse embeddability into a Hilbert space. 
  
 We now describe the main results of this paper.
  
\begin{enumerate}
\popcounter
\item \label{ThmMAIN} Properties    \eqref{I.II}, \eqref{I.IsoAlg}, and \eqref{I.IsoDiag}  are equivalent for all uniformly locally finite coarse spaces (Theorem \ref{T.ThmMAIN}). 
\item 
\label{ThmRigIsoImpliesCoarseEqui}
\eqref{I.IsoRig} implies \eqref{I.I} for 
 uniformly locally finite metric spaces (Theorem~\ref{T.ThmRigIsoImpliesCoarseEqui}). 
\item \label{ThmRigIsoImpliesCoarseEquiMARTIN}
More generally, it is consistent with ZFC that \eqref{I.IsoRig} implies \eqref{I.I} for 
 uniformly locally finite coarse spaces whose coarse structures are generated by less than $2^{\aleph_0}$ subsets
 (see Theorem~\ref{T.ThmRigIsoImpliesCoarseEqui} for a stronger result).
\item \label{ThmIsoImpliesRigIso}\label{I.CorHilbertEmbRigidity}
\eqref{I.Iso} implies  \eqref{I.IsoRig} for 
uniformly locally finite metric spaces  which coarsely embed into a Hilbert space
(Theorem~\ref{T.CorHilbertEmbRigidity}). 
\item \label{ThmIsoImpliesRigIsoGhost}
\eqref{I.Iso} implies  \eqref{I.IsoRig} for  
uniformly locally finite metric spaces 
in which all the  ghost projections are compact
(Theorem~\ref{C.ThmIsoImpliesRigIsoGhost}). 
 More precisely, if an isomorphism is not 
rigidly implemented, then at least one of the uniform Roe algebras  
has a Cartan masa which contains non-compact ghost projections (Theorem~\ref{T.ThmIsoImpliesRigIsoGhost}). 
\end{enumerate}

By \eqref{I.CorHilbertEmbRigidity} and \eqref{ThmRigIsoImpliesCoarseEqui}, we have the following. 

\begin{cor}\label{CorHilbertEmbRigidity}
Let $(X,d)$ and $(Y,\partial)$ be
  uniformly locally finite metric spaces which coarsely embed into a Hilbert space.   If $\cstu(X)$ and $\cstu(Y)$ are    isomorphic, then $(X,d)$ and $(Y,\partial)$ are coarsely equivalent. \qed
\end{cor}

By \eqref{ThmIsoImpliesRigIsoGhost} and \eqref{ThmRigIsoImpliesCoarseEqui}, we have the following. 

\begin{cor}
Let $(X,d)$ and $(Y,\partial)$ be
  uniformly locally finite metric spaces such  that all the  ghost projections in $\cstu(X)$ and $\cstu(Y)$ are compact. If $\cstu(X)$ and $\cstu(Y)$ are    isomorphic, then $(X,d)$ and $(Y,\partial)$ are coarsely equivalent. \qed
\end{cor}

We do not know whether a Cartan masa  can contain a non-compact  ghost projection (see the conclusion of \eqref{ThmIsoImpliesRigIsoGhost}); see however Example~\ref{Ex.Ghost}.

A uniformly locally finite metric space which coarsely embeds into a Hilbert space but  does not have property A was constructed in \cite{ArzhantsevaGuentnerSpakula2012}, Theorem 1.1. Although this space does not have property A, all ghost projections in $\cstu(X)$ are compact (see \cite{RoeWillett2014}, Theorem 1.4). In particular, it follows from Theorem 1.1 of \cite{Yu2000} that this space satisfies the coarse Baum--Connes conjecture.

Our main technical contribution is that for uniformly locally finite metric spaces $X$ and $Y$, any isomorphism $\cstu(Y)\to\cstu(X)$ must satisfy a certain kind of automatic uniform approximation for ``banded'' operators. We believe that this is likely to have other applications,  especially to 
banded matrices in operator theory and mathematical physics. We proceed to give an informal description of this result,  and  refer to Definition \ref{Def.CoarseLike} and Theorem \ref{T.Approximation} for precise statements and generalization for nonmetrizable spaces. Let $(X,\cE)$ and $(Y,\cF)$ be metrizable uniformly locally finite coarse spaces. 
We prove that any    isomorphism $\cstu(Y)\to \cstu(X)$ 
must satisfy the following  `coarse-like' property: roughly  speaking, given such    isomorphism, we show that for every $\eps>0$ there exists an assignment $F\in\cF\mapsto E_F\in \cE$ so  that this    isomorphism takes operators   supported in $F$ to operators supported in $E_F$ up to a error of $\eps$.  In the case of the algebraic uniform Roe algebra, the error may be taken to be zero and the metrizability assumption can be omitted (Theorem~\ref{ThmAlgebraicMain}).

This paper is organized as follows.  In Section \ref{SectionBackground}, we present all the necessary 
definitions and background for these notes. Section \ref{SectionSpatially} is dedicated to show that Lemma 3.1 of \cite{SpakulaWillett2013} can be generalized to non-connected  coarse spaces. 
In Section \ref{SectionIsoRoeAlg}, we prove the `coarse-like' property mentioned above which gives us (\ref{ThmRigIsoImpliesCoarseEqui}) and (\ref{ThmRigIsoImpliesCoarseEquiMARTIN}). We can also prove rigidity of uniform Roe algebras for a subclass  of uniformly locally finite coarse spaces which we call  spaces with  \emph{small partitions}. This is done in Section~\ref{SubsectionUnformRoeAlgDisjointUnion} and  we refer the reader to Definition \ref{DefinitionUniformlyMetrizable}  and Theorem \ref{ThmPartitionUniformlyMetrizable} for precise statements. In Section~\ref{SectionGhostlyMasas}, we look at Cartan masas (Definition \ref{DefiCartanMasa}) and show (\ref{ThmIsoImpliesRigIsoGhost}) above. 
 In Section \ref{SectionRigidityHilbert}, we show that rigidity holds for metric spaces which coarsely embed into a Hilbert space ((\ref{ThmIsoImpliesRigIso}) above) and, in Section \ref{SectionIsoAlgebraicRoeAlg}, we deal with rigidity of algebraic uniform Roe algebras and prove (\ref{ThmMAIN}).

\begin{remark}
The results and techniques introduced in this paper have already been successfully applied to a variety of loosely related subjects:   
rigidity for isomorphisms between the quotients of uniform Roe algebras 
by the compact operators (\cite{BragaFarahVignati2018}), embeddings between uniform Roe algebras (\cite{BragaFarahVignati2019}), Banach algebra isomorphisms between uniform Roe algebras  (\cite{BragaVignati2019}), rigidity results for metric spaces satisfying the coarse Baum--Connes conjecture with coefficients  (\cite{BragaChungLi2019}), and Cartan masas of $\cstu(X)$ (\cite{WhiteWillett2017}).
\end{remark}

\section{Background}\label{SectionBackground}

We start this section by giving the basic definitions about coarse spaces. For a detailed account of that,   we refer the reader to \cite{RoeBook} and \cite{RosendalBook}.

\subsection{Coarse spaces}\label{SubsectionCoarseSpaces}
%
 Let $X$ be a set. Given subsets $E,F\subset X\times X$, we define
\[
E^{-1}=\{(x,y)\in X\times X\mid (y,x)\in E\}
\]
and
\[E\circ F=\{(x,y)\in X\times X\mid \exists z\in X\ \ \text{with}\  \ (x,z)\in E\ \ \text{and}\ \ (z,y)\in F\}.\]
We say that $E$ is \emph{symmetric} if $E=E^{-1}$. For each $n\in\N$, define $E^{(n)}$ recursively as follows. Let $E^{(1)}=E$ and $E^{(n+1)}=E\circ E^{(n)}$, for all $n\geq 1$.  

\begin{defi}\label{DefiCoarseSpace}
Let $X$ be a set. 
A collection $\cE$ of subsets of $X\times X$ is called  a \emph{coarse structure} if
\begin{enumerate}[(i)]
\item $\Delta_X\coloneqq \{(x,x)\in X \times X\mid x\in X\}\in \mathcal{E}$,
\item $E\in\mathcal{E}$ and $F\subset E$ implies $F\in\mathcal{E}$,
\item $E\in \mathcal{E}$ implies $E^{-1}\in\mathcal{E}$, 
\item $E,F\in\mathcal{E}$ implies $E\cup F\in\mathcal{E}$, and 
\item $E,F\in \mathcal{E}$ implies $E\circ F\in\mathcal{E}$.
\end{enumerate}
The elements of $\mathcal{E}$ are called \emph{entourages} and the pair $(X,\mathcal{E})$  is called a \emph{coarse space}.
\end{defi}

Let $X$ be a set, $\cE\subset \cP(X\times X)$ and $A\subset X$ be a subset. We define  
\[\mathcal{E}_A=\{E\cap A\times A\mid E\in\mathcal{E}\}.\]
If $\cE$ is a coarse structure on $X$, then $\cE_A$ defines a coarse structure on $A$. A coarse space $(X,\cE)$ is called \emph{connected} if $\{(x,y)\}\in\cE$, for all $x,y\in X$. Since $\{(x,y)\}\in\mathcal{E}$ defines an equivalence relation on $X$, we can always write $X=\bigsqcup_{j\in J}X_j$, where each $(X_j,\mathcal{E}_{X_j})$ is connected and $X_j\cap X_i=\emptyset$, for all $j\neq i$. The subsets $(X_j)_{j\in J}$ are called the \emph{connected components of $X$}.

 Given a set $X$ and a family of subsets $\{E_i\}_{i\in I}\subset \mathcal{P}(X\times X)$, the intersection of all the coarse structures on $X$ containing the family $\{E_i\}_{i\in I}$, say $\cE$, is still a coarse structure and it is called the \emph{coarse structure generated by $\{E_i\}_{i\in I}$}. The family $\{E_i\}_{i\in I}$ is called a set of \emph{generators} of  $\cE$. We say that a coarse structure $\mathcal{E}$ on $X$ is \emph{countably generated} if it is generated by a countable family of subsets of $X\times X$.

Let $(X,d)$ be a metric space. For each $r\geq 0$, let 
\[E_r=\{(x,y)\in X\times X\mid d(x,y)\leq r\}.\]
We call the  the coarse structure generated by $\{E_r\}_{r\geq 0}$ the \emph{bounded coarse structure of $(X,d)$} and we denote it by $\mathcal{E}_d$. Clearly, $\mathcal{E}_d$ is countably generated. 

A coarse space $(X,\mathcal{E})$ is \emph{metrizable} if there exists some metric $d$ on $X$ such that $\mathcal{E}$ is the bounded coarse structure of $(X,d)$. 
A connected coarse structure $(X,\mathcal{E})$ is metrizable if and only if $\mathcal{E}$ is countably generated
(see \cite{RoeBook}, Theorem 2.55).

For $E\subset X\times X$ and $x\in X$ let  
\[
E_x=\{y\in X\mid (x,y)\in E\}\ \ \text{and}\ \ E^x=\{y\in X\mid (y,x)\in E\}.
\]
\begin{defi}
A coarse space $(X,\mathcal{E})$ is called \emph{uniformly locally finite} if  
\[\sup_{x\in X} |E_x|<\infty,\]
for all $E\in\mathcal{E}$. If  $(X,d)$ is a metric space and $\cE_d$ is its bounded coarse structure, we say that $(X,d)$ is a \emph{uniformly locally finite metric space} if $(X,\cE_d)$ is a uniformly locally finite coarse space.
\end{defi}

\begin{remark}
We notice that a uniformly locally finite metric space $(X,d)$ is usually called a \emph{metric space with bounded geometry} in the literature. However, the common definition of bounded geometry  for general coarse spaces (see \cite{RoeBook}, Chapter 3, Section 3.1) is not the generalization we need for these notes. Precisely, what we need is the idea of \emph{uniformly discrete coarse spaces} defined in \cite{RoeBook}, Definition 3.24. Since the terminology \emph{uniformly locally finite}  is also used by other authors (e.g.,  \cite{Sako2012}, \cite{Sako2013} and \cite{RosendalBook}), we chose to use this less common terminology (even for metric spaces). 
\end{remark}

Let $(X,\mathcal{E})$ be a coarse space and let $E\in\mathcal{E}$. A subset $X'$ of $X$ is called \emph{$E$-separated} if $(x,y)\not\in E$, for all distinct $x,y\in X'$. The following
lemma is a reformulation of  Lemma 2.7(a), of \cite{SkandalisTuYu2002}. For the convenience of the reader, we include a proof.

\begin{prop}\label{partition}
Let $(X,\mathcal{E})$ be a coarse structure. Let $E\in\mathcal{E}$ be such that  $n\coloneqq \sup_{x\in X}|E_x|< \infty$. Then there exists a partition 
\[X=X_1\sqcup\ldots\sqcup X_n,\]
such that $X_i$ is $E$-separated, for all $i\in \{1,\ldots,n\}$. In particular,  if $(X,\mathcal{E})$ is a uniformly locally finite  coarse space, then such partition exists for all $E\in\mathcal{E}$.
\end{prop}

\begin{proof}
Fix $E\in\mathcal{E}$ as in the proposition. Without loss of generality, we assume that $E$ is symmetric and that $\Delta_X\subset E$. By Zorn's lemma, we can pick a maximal $E$-separated subset  $X_1\subset X$. Assume $X_i$ has being defined, for $i\geq 1$. If \[X\setminus \bigcup_{n=1}^iX_n\neq \emptyset,\] let $X_{i+1}\subset X\setminus \bigcup_{n=1}^iX_n$ be a maximal $E$-separated subset. If the proposition does not hold, this defines a finite sequence $(X_i)_{i=1}^{n+1}$ of nonempty pairwise disjoint   subsets such that, for all $i\leq n+1$, $X_{i}\subset X\setminus \bigcup_{n=1}^{i-1}X_n$ and $X_i$ is  a maximal $E$-separated subset of $X\setminus \bigcup_{n=1}^{i-1}X_n$. Pick $x\in X_{n+1}$. For each $i\leq n$, we can pick $x_i\in X_i$ such that $(x,x_i)\in E$. Therefore,
\[\{x\}\cup\{x_i\mid j\in\{1,\ldots,n\}\}\subset E_x.\]
This gives us that $|E_x|\geq n+1$;  contradiction.
\end{proof}

\begin{defi}
Let $(X,\mathcal{E})$ and $(Y,\mathcal{F})$ be  coarse spaces.  If $Z$ is any set, 
then maps   $f:Z\to X$ and $g:Z\to X$ are said to be  \emph{close} if
\[
\{(f(x),f(x'))\in X\times X\mid x,x'\in Z\}\in \mathcal{E}.
\]   
\begin{enumerate}[(i)]
\item A map  $f:X\to Y$  
 is called \emph{coarse} if for all $E\in\mathcal{E}$ there exists $F\in \mathcal{F}$ such that $(x,x')\in E$ implies $(f(x),f(x'))\in F$. 
\item A coarse map $f:X\to Y$  is a \emph{coarse embedding} if 
 for all $F\in \mathcal{F}$ there exists $E\in\mathcal{E}$ such that $(x,x')\not\in E$ implies $(f(x),f(x'))\not\in F$.
\item A coarse map $f:X\to Y$ 
 is a \emph{coarse equivalence} if  there exists a coarse map $g:Y\to X$ such that $g\circ f$ and $f\circ g$  are close to $\Id_X$ and $\Id_Y$, respectively.\footnote{Notice that a coarse equivalence is automatically a coarse embedding.}
 \end{enumerate}
\end{defi}
Equivalently, $f$ is a coarse equivalence if it is a coarse embedding and it is \emph{cobounded}, i.e., there exists $F\in\mathcal{F}$ such that 
\[Y=\{y\in Y\mid \exists x\in X,\  (f(x),y)\in F\}.\]

\subsection{Uniform Roe algebras}\label{SubsectionUnformRoeAlg}
We denote the algebra of bounded linear operators on a Hilbert space $H$   
 by $B(H)$. 

\begin{defi}
Let $(X,\mathcal{E})$ be a  coarse space. The \emph{algebraic uniform Roe algebra} $\cstu[X]$ is defined by setting
\[\cstu[X,\cE]=\Big\{T\in B(\ell_2(X))\mid \exists E\in\cE,\ \forall (x,x')\not\in E,\  \langle T\delta_x,\delta_{x'}\rangle=0\Big\},\]
and the \emph{uniform Roe algebra} $\cstu(X,\cE)$ is defined as the norm closure of $\cstu[X]$ in $B(\ell_2(X))$. Clearly,   $\cstu[X,\cE]$ is a    algebra and  $\cstu(X,\cE)$ is a \cstar -algebra. We omit $\cE$ whenever it is clear from the context and write $\cstu[X]$ and $\cstu(X)$ for $\cst[X,\cE]$ and $\cstu(X,\cE)$, 
respectively.   
\end{defi}

Since we will only work with uniformly locally finite coarse spaces $(X,\cE)$, it is worth noticing that, for any such $(X,\cE)$, the following holds. Let  $T=(T_{xy})_{x,y\in X}$ be a family of complex numbers satisfying 
\begin{enumerate}[(i)]
\item $\sup_{x,y\in X}|T_{xy}|< \infty$, and
\item there exists $E\in \mathcal{E}$ such that $T_{xy}=0$ if $(x,y)\not\in E$.
\end{enumerate}
The family $T$ naturally induces a bounded operator on $\ell_2(X)$, which  we still denote  by $T$. The set of all such operators coincides with $\cstu[X]$ and, for any such  $T=(T_{xy})_{x,y\in X}$, 
it holds that \[
\|T\|\leq \sup_{x\in X}|E^x|\cdot \sup\{|T_{xy}|\mid x,y\in X\},
\]
for $E\in \mathcal{E}$ as in (ii).

For a set  $X$ the Hilbert space $\ell_2(X)$ has  a  standard orthonormal 
 basis, which we denote by  $(\delta_x)_{x\in X}$. The \emph{support} of $T\in B(\ell_2(X))$ is 
\[
\supp(T)=\{(x,x')\in X\times X\mid \langle T\delta_x,\delta_{x'}\rangle\neq 0\}.
\]
By identifying  $\ell_\infty(X)$ with the subspace of all $T\in B(\ell_2(X))$ 
such that $\supp(T)\subseteq \Delta_X$ we have the inclusion $\ell_\infty(X)\subset \cstu(X)$.

The algebra $\ell_\infty(X)$ is a maximal abelian subalgebra (or shortly, masa) with special properties, 
captured by the following definition
 (\cite{WhiteWillett2017},  Proposition 3.1).

\begin{defi}\label{DefiCartanMasa}
A \cstar-subalgebra   
$B$ of  a \cstar -algebra 
$A$ is a \emph{Cartan masa} if
\begin{enumerate}[(i)]
\item $B$ is a maximal abelian self-adjoint subalgebra (i.e., a \emph{masa}) of $A$,
\item $B$ contains an approximate unit for $A$,
\item the \emph{normalizer} of $B$ in $A$ defined as 
\[\cN_A(B)=\{a\in A\mid aBa^*\cup a^*Ba\subset B\}\]
generates $A$ as a \cstar -algebra, and 
\item there is a faithful conditional expectation from $A$ onto  $B$.
\end{enumerate}
\end{defi}

Given $x,x'\in X$, we define an operator $e_{xx'}\in B(\ell_2(X))$ by letting
\[e_{xx'}(\delta_z)=\langle \delta_z, \delta_{x'}\rangle \delta_x,\]
for all $z\in Z$. For $A\subset X$, let $\chi_A=\sum_{x\in A} e_{xx}$. Clearly, $\chi_A$ belongs to $\ell_\infty(X)$. Given $T\in B(\ell_2(X))$ and $x,x'\in X$, we let 
\[T_{xx'}=e_{xx}Te_{x'x'}.\]
So,  if $\{(x,x')\}\in\mathcal{E}$, we have that $T_{xx'}\in \cstu[X]$. 

We state the following trivial lemma for future reference. 

\begin{lemma}\label{LemmaSOTConv}
Suppose $(X,\mathcal{E})$ is a uniformly locally finite coarse space and  
$(T_{j})_{j\in J}$ is  a uniformly bounded family of operators with disjoint supports 
 such that  $E=\bigcup_{j\in J} \supp(T_j)$ 
belongs to $\cE$. Then the series $\sum_{j\in J}T_J$ converges in the strong operator topology 
to an element in $\cstu[X]$ with support contained in~$E$. \qed
\end{lemma}

We conclude this subsection with the definition of ghost operators in uniform Roe algebras.

\begin{defi}\label{DefiGhost}
Let $(X,\mathcal{E})$ be a uniformly locally finite coarse space. An operator $T\in \cstu(X)$ is a \emph{ghost} if $\langle T\delta_x,\delta_{x'}\rangle\to 0$ as $(x,x')\to \infty$ on $X\times X$, i.e., if for all $\eps>0$ there exists a finite set $A\subset X$ such that 
\[|\langle T\delta_x,\delta_{x'}\rangle|< \eps,\]
for all $x,x'\in X\setminus A$.
\end{defi}

The set of  all ghost operators of a  uniform Roe algebra forms an ideal which contains all compact operators. For more on the class of ghost operators, we refer to \cite{ChenWang2004}, \cite{ChenWang2005} and \cite{RoeBook}.

\section{Spatially implemented isomorphisms}\label{SectionSpatially}

In this section, we show that 
a minor modification of the proof of  \cite[Lemma~3.1]{SpakulaWillett2013} 
gives that 
 any    isomorphism between uniform Roe algebras is spatially implemented, thus showing
  \eqref{SpatiallyImplemented} from the introduction. 

\begin{lemma}\label{LemmaSpakulaWillett} 
Every isomorphism between uniform Roe algebras associated with coarse spaces
is spatially implemented. 
Every isomorphism between algebraic uniform Roe algebras associated with coarse spaces
is spatially implemented. 
\end{lemma} 

We only need to show that the assumptions that the spaces be connected and metrizable
used in  \cite[Lemma~3.1]{SpakulaWillett2013} are not necessary for the conclusion. This requires only 
a little analysis of the center of an (algebraic)  uniform Roe algebra.

Suppose  $(X,\mathcal{E})$ is a   coarse space
with  connected components $X_i$, for $i\in I$. 
The space $\ell_2(X)$ can be naturally identified with  the Hilbert sum
$\bigoplus_{i\in I} \ell_2(X_i)$. 
Denote the projection from $\ell_2(X)$ to $\ell_2(X_i)$ by $1_{X_i}$.  
A \emph{corner} of a \cstar-algebra $A$ is a subalgebra of the form $PAP$ for some nonzero 
projection $P$ in $A$.
A projection $P$ in a \cstar-algebra $A$ is called a \emph{scalar projection of $A$} if the corner 
$PAP$ is isomorphic to $\C$.  
The center of an algebra $A$ is denoted  $Z(A)$. 

\begin{lemma} \label{LCentre}
Let $(X,\mathcal{E})$ be a   coarse space
with  connected components $X_i$, for $i\in I$. Then, the projections   
$(1_{X_i}:\ell_2(X)\to \ell_2(X_i))_{i\in I}$ are 
the only  scalar  projections of  $Z(\cstu(X))$. 
They are also the only scalar projections  of $Z(\cstu[X])$. 
\end{lemma} 

\begin{proof} If $X$ is connected, then $\cstu[X]$ includes all finite rank operators 
and therefore $Z(\cstu[X])=Z(\cstu(X))=\C\cdot 1$, so the statement follows. 
In the general case, 
via the  identification from the previous paragraph,  $\cstu(X)$ is identified with $\prod_{i\in I} B(\ell_2(X_i))$. Therefore, we can see that $Z(\cstu[X])=Z(\cstu(X))=\oplus_{i\in I}\mathbb{C}\cdot 1_{X_i}$ and the statement follows.
\end{proof}

\begin{proof}[Proof of Lemma~\ref{LemmaSpakulaWillett}]
Suppose  $(X,\mathcal{E})$ and $(Y,\mathcal{F})$ are  coarse spaces
with connected components $X_i$, for $i\in I$, and $Y_j$, for $j\in J$, 
and that   $\Phi:\cstu(X)\to \cstu(Y)$
is an isomorphism.  
Lemma~\ref{LCentre} implies that there exists a bijection $f\colon I\to J$ 
such that $\Phi(1_{X_i})=1_{Y_{f(i)}}$. 

It remains to prove that the isomorphism between $\cstu(X_i)$ and $\cstu(Y_{f(i)})$
is implemented by a unitary for every $i$.  

Note that $\Phi$ sends scalar projections of $\cstu(X)$ to $\cstu(Y)$. 
Fix $i$. 
The algebra of compact operators $K(\ell_2(X_i))$ is  
equal to the ideal  
of $\cstu(X_i)$ generated by its scalar projections, 
and 
the algebra of compact operators $K(\ell_2(Y_{f(i)}))$ is  
equal to the ideal  
of $\cstu(Y_{f(i)})$ generated by its scalar projections. 
Therefore $\Phi(K(\ell_2(X_i))=K(\ell_2(Y_{f(i)})$. 
Since an isomorphism between two algebras of compact operators is implemented by a unitary 
unique up to the multiplication by a scalar, a unitary $U_i\colon \ell_2(X_i)\to \ell_2(Y_{f(i)})$
implements the restriction of $\Phi$ to $\cstu(X_i)$. This concludes the proof in the case of 
uniform Roe algebras. 

An analogous argument shows the analogous statement holds for $\cstu[X]$. 
\end{proof}

\begin{cor}\label{LemmaRank}
An  isomorphism between uniform Roe algebras 
carries operators of rank $n$ to operators of rank $n$ for all $n\in \N$.  
It also carries operators with orthogonal images to operators with orthogonal images. \qed
\end{cor}

\begin{cor}\label{CorCardinality}
Let $(X,\cE)$ and $(Y,\cF)$ be coarse spaces with    isomorphic uniform Roe algebras. Then $|X|=|Y|$.
\end{cor}

\begin{proof}
This follows straightforwardly from Lemma \ref{LemmaSpakulaWillett}, the fact that an unitary isomorphism $U:\ell_2(X)\to \ell_2(Y)$ must take an orthonormal basis of $\ell_2(X)$ to an orthonormal basis of $\ell_2(Y)$, and the fact that an orthonormal basis of $\ell_2(X)$ has cardinality $|X|$. 
\end{proof}

\section{Rigid isomorphism between uniform Roe algebras and Baire category}\label{SectionIsoRoeAlg}

The goal of this section is to prove    \eqref{ThmRigIsoImpliesCoarseEqui} and \eqref{ThmRigIsoImpliesCoarseEquiMARTIN} 
(Theorem~\ref{T.ThmRigIsoImpliesCoarseEqui})
 from the introduction. 
The main step in the proof of this result is Theorem \ref{T.Approximation} below, which 
 is a strengthening  of Lemma 3.2 of \cite{SpakulaWillett2013}.

\begin{quote} 
Readers interested only in metric  spaces will lose nothing by 
assuming all the spaces are metrizable throughout this section.\footnote{This is because the 
 Baire category theorem   implies that all  metric spaces are, in the 
 terminology introduced below,  small.  
It is hard to resist quoting  from \cite{shelah1992cardinal}: 
``To these we have nothing to say at all, beyond a reasonable request that they refrain from using the countable additivity of Lebesgue measure.''}
 \end{quote} 
Other readers may want to consult \cite[Section~III]{Ku:Set} for more details on infinitary combinatorics and  
Martin's Axiom in particular. Define 
\[\D=\{z\in\C\mid |z|\leq 1\}\]
and endow $\D$ with its usual metric. For a set $J$ consider the space $\D^J$ with the product topology. 
This is a compact Hausdorff space. Let $\text{cov}(\D^J)$ denote the minimal cardinality 
of a family of nowhere dense subsets of $\D^J$ that covers the space. 
The following lemma collects some well-known results.

\begin{lemma}\label{L.cov}  Suppose that  $\kappa$ and $\mu$ are infinite cardinals.  
\begin{enumerate}[(i)]
\item \label{cov.0} If $\text{cov}(\D^\kappa)\leq \kappa$ then $\kappa$ is uncountable.  
 \item \label{cov.2} If $\kappa<\mu$ then $\text{cov}(\D^\kappa)\geq  \text{cov}(\D^\mu)$. 
 \item \label{cov.3} If $\kappa\geq 2^{\aleph_0}$ then $\text{cov}(\D^\kappa)\leq \kappa$. 
\item \label{cov.1} Martin's Axiom for $\kappa$ dense sets, MA$_\kappa$,  
 implies that  $\cov(\D^\mu)>\kappa$ for all~$\mu$. 
 \item \label{cov.4} It is consistent with ZFC that $\text{cov}(\D^\mu)>\kappa$
 for all $\kappa<2^{\aleph_0}$ and all~$\mu$. 
\end{enumerate}
\end{lemma} 

\begin{proof} 
\eqref{cov.0} is the Baire category theorem. 

For \eqref{cov.2}, note that if $\D^\mu$ is homeomorphic to $\D^\kappa\times \D^\mu$, 
and that if $F$ is nowhere dense in $\D^\kappa$ then $F\times \D^\mu$ is nowhere dense in $\twom$. 

For \eqref{cov.3}, cover $\D^\N$ by the singletons and apply \eqref{cov.2}. 

The Tychonoff product of any family of separable spaces has the countable chain condition (ccc; see \cite[Definition~III.2.1]{Ku:Set}) by 
\cite[Theorem~III.2.8]{Ku:Set}. 
Therefore  the space $\D^\mu$ has the countable chain condition and 
\eqref{cov.1}  is a consequence of \cite[Lemma~III.3.18]{Ku:Set}. 
For \eqref{cov.4}, use \eqref{cov.3} and the consistency of Martin's Axiom with ZFC (\cite{Ku:Set}). 
\end{proof}

The following (nonstandard)  terminology will be useful.

\begin{defi} \label{Def.Small}
Let $\kappa$ and $\mu$ be cardinals. The pair $(\kappa, \mu)$ is called \emph{small} if $\text{cov}(\D^\mu)>\kappa$.    
  A coarse space $(X,\cE)$ is said to be \emph{small} if the pair $(|\cE|_\text{min},|X|)$ is small, where $|\cE|_\text{min}$ is defined as the minimal cardinality of a set of generators of~$\cE$. 
\end{defi} 

Since for every metrizable space $(X,\cE)$ we have that  $|\cE|_{\min}$ is countable, the Baire category theorem implies that every metrizable coarse space is small. By Lemma~\ref{L.cov}(iv), if $|\cE|_{\text{min}}<2^{\aleph_0}$, then Martin's Axiom for $|\cE|_{\text{min}}$ dense sets implies that $(X,\cE)$ is small.  On the other hand, by Lemma~\ref{L.cov}(iii), if $|\cE|_\text{min}\geq 2^{\aleph_0}$ and $\text{cov}(\D^{X})\leq \text{cov}(\D^{|\cE|_\text{min}})$, then $(X,\cE)$ is not small. In particular, a metric space with $|\cE|_\text{min}= 2^{\aleph_0}$ is not small.

\begin{defi} \label{Def.CoarseLike} 
Suppose $(X,\cE)$ is a coarse space, $\e>0$, and $E\in \cE$. An operator $T\in B(\ell_2(X))$
can be  \emph{$\e$-$E$-approximated} if 
there exists $S\in \cstu(X)$ such that $\supp(S)\subseteq E$ and $\|T-S\|\leq\e$. 
We say that $S$ is an \emph{$\e$-$E$-approximation} to $T$. 
\end{defi} 

\begin{theorem} \label{T.Approximation} 
Suppose that   $(X,\cE)$ and $(Y,\cF)$ are  small and  uniformly locally finite coarse spaces. Let  
  $\Phi\colon \cstu(Y)\to \cstu(X)$ be am    isomomorphism. 
 Then for every $F\in \cF$ and $\e>0$ there exists $E\coloneqq E(F,\e)\in \cE$ 
 such that $\Phi(T)$ can be  $\e$-$E$-approximated for all $T\in \cstu(Y)$
 such that $\supp(T)\subseteq F$ and $\|T\|\leq 1$. 
 \end{theorem} 
  
The  proof of Theorem~\ref{T.Approximation} was inspired by 
the canonical Ramsey Theory (see e.g., \cite{promel1985canonical}). 
For the convenience of the reader, we give a self-contained proof  after a few elementary lemmas
and some bad news. 

\begin{lemma} \label{L.eEapprox}
Suppose that $(X,\cE)$ is a coarse space, $\e>0$, and $E\in \cE$. 
If operators $T_1$ and $T_2$ are $\e$-$E$-approximated, then the operator $T_1+T_2$ can be $2\e$-$E$-approximated.\qed
\end{lemma}


\begin{lemma} \label{L.eEapprox2}
Suppose $(X,\cE)$ is a uniformly locally finite coarse space, $\e>0$, and $E\in \cE$. 
Let   $T$ be an operator in $\cstu(X)$ and $P$ be a finite rank projection in $\ell_\infty(X)$. The following holds.

\begin{enumerate}[(i)]
\item If $T$ is $\eps$-$E$-approximated, so is $TP$.
\item If $TP$ is  $(\e+\delta)$-$E$-approximated for all $\delta>0$, then $TP$ is  $\e$-$E$-approximated.
\end{enumerate} 
\end{lemma} 

\begin{proof}
Since $\supp(TP)\subset \supp(T)$ and $\|P\|\leq 1$, (i) follows. For (ii), pick $S_n\in\cstu(X)$  such that $\|TP-S_n\|\leq \eps+1/n$ and $\supp(S_n)\subset E$, for all $n\in\N$. Let 
\[X'=\{x\in X\mid P\delta_x\neq 0\}\ \text{ and }\ X''=\{x\in X\mid (\exists x'\in  X') (x',x)\in E\},\]
and notice that both $X'$ and $X''$ are finite. Therefore, since $S_nP$ can be naturally identified with operators from $\ell_2(X')$ to $\ell_2(X'')$, by going to a subsequence, we can assume that $(S_nP)_{n\in\N}$ converges to some $S\in B(\ell_2(X))$ in norm. As $\supp(S_nP)\subset E$, for all $n\in\N$, $\supp(S)\subset E$. Clearly, $\|TP-S\|\leq \eps$. 
\end{proof}

\begin{lemma} \label{L.eE2}
Suppose $(X,\cE)$ is a coarse space, $\e>0$, and $E\in \cE$. 
Also suppose that  $(P_j)_{j\in J}$ is an increasing net  
of  projections in $\ell_\infty(X)$ converging to $1$ strongly. 
If  $T\in B(\ell_2(X))$ cannot be $\e$-$E$-approximated, 
then $TP_\lambda$ cannot be $\e$-$E$-approximated 
for all large enough $\lambda$. 
\end{lemma} 

\begin{proof} Suppose otherwise. For every $j\in J$ fix an  $\e$-$E$-approximation $S_j$ to $T P_j$. 
Then $\|S_j\|\leq \|T\|+\e$ for all $j\in J$. Since the norm-bounded balls of $B(\ell_2(X))$ 
are compact in the weak operator topology, 
by going to a subnet if necessary we may assume that the  $(S_j)_{j\in J}$ converges to some
$S\in B(\ell_2(X))$
in the weak operator topology. 
Clearly,  $\supp(S)\subseteq \bigcup_{j\in J} \supp(S_j)$. So,  $S\in \cstu(X)$ and, by our assumption,  
$\|T-S\|>\e$. Choose unit vectors  $\xi$ and $\eta$  in $\ell_2(X)$ such that $|\langle (T-S)\xi,\eta\rangle|>\e$. 
Since $\lim_{j\in J} P_j \xi=\xi$ in norm, we have
 $\lim_{j\in J} |\langle (T P_j -S_j) \xi, \eta\rangle|>\e$
 contradicting the assumption that $S_j$ is a $\eps$-$E$-approximation of $TP_j$. 
 \end{proof} 

\begin{lemma} \label{L.eE.cpct}
Suppose $(X,\cE)$ is a coarse space and $K\subseteq \cstu(X)$ is compact in the norm 
topology. Then for every $\e>0$ there exists $E\in \cE$ such that every $T\in K$ can be  
$\e$-$E$-approximated. 
\end{lemma}

\begin{proof} If $K$ is finite then this is true because $\cE$ is directed. 
For the general case, fix a finite $\e/2$-net $K_0$ in $K$ and find $E$ such that every $T\in K_0$
can be $\e/2$-$E$-approximated. For  $T'\in K$ fix $T\in K_0$ such that $\|T'-T\|<\eps$.  Lemma \ref{L.eEapprox} implies that $T'=T'-T+T$ can be $\eps$-$E$-approximated. Since $T'\in K$ was arbitrary, the conclusion follows.  
\end{proof}

In what follows, it will be convenient to write $\bar\lambda$ for $(\lambda_j)_{j\in J}\in \D^J$. 

\begin{lemma} \label{L.Approx} Suppose  $(X,\cE)$ is a small 
and  uniformly locally finite coarse space. Suppose $(T_j)_{j\in J}$ is a family of   operators in $\cstu(X)$ such 
that 
for every $\bar\lambda\in \D^J$ the series $\sum_{j\in J} \lambda_j T_j$ strongly converges 
to an operator $T_{\bar\lambda}\in \cstu(X)$. 
 Then for every $\e>0$ there exists $E\in \cE$ such that 
$T_{\bar\lambda}$ can be $\e$-$E$-approximated for all $\bar\lambda\in\D^J$. 
\end{lemma}

\begin{proof} 
Without loss of generality,  assume that $X$ is infinite and that $T_j\neq 0$ for all $j$. Hence,
it follows that $|J|\leq |X|$. Otherwise, by a counting argument 
there is a pair $(x,x')$ in $X^2$ such that 
the set $\{j\in J\mid (T_j)_{xx'}\neq 0\}$ is of cardinality $|J|$. Since  $X$ is infinite,  $J$ is uncountable and 
therefore for some $\e>0$ the set $\{j\in J\mid |(T_j)_{xx'}|\geq \e\}$ is uncountable. This clearly contradicts the assumption that $\sum_{j\in J} \lambda_j T_j$ 
converges strongly to some operator $T$. 

 For each finite $I\subset J$, write 
\[\cZ_I=\{\bar\lambda\in\D^J\mid \forall j\in I, \ \lambda_j=0\}\text{ and }\cY_I=\{\bar\lambda\in \D^J\mid \forall j\not\in I,\  \lambda_j=0\}.\] 
Assume that the conclusion of the lemma fails. Then the following holds. 
\begin{enumerate}
\popcounter
\item\label{*.eE}  $(\exists \e>0)(\forall E\in \cE) (\exists \bar\lambda\in\D^J)  T_{\bar\lambda}$ is not $\e$-$E$-approximated. 
\pushcounter
\end{enumerate}
This implies a  stronger condition. In what follows, the symbol $\forall_{\text{fin}}$ abbreviates `for all finite'.  
\begin{enumerate}
\popcounter
\item\label{**.eE}  $(\exists \e'>0)(\forall E\in\cE)(\forall_{\text{fin}}I\subset J) (\exists \bar\lambda\in \cZ_I)  T_{\bar\lambda}$ is not $\e'$-$E$-approximated. 
\pushcounter
\end{enumerate}
To prove this, suppose \eqref{**.eE} fails for $\e'>0$ and fix $E\in\cE$ and a finite $I\subset J$ such that for all $\bar\lambda\in \cZ_I$
the operator $T_{\bar\lambda}$ can be $\e'$-$E$-approximated. 
For each $\bar s\in \cY_I$  the operator $T_{\bar s}$ belongs to $\cstu(X)$.
Since the function $\bar s\in \cY_I\mapsto T_{\bar s}\in \cstu(X)$ 
is norm-continuous and $\cY_I$ is, being homeomorphic to $\D^I$, compact, 
by Lemma~\ref{L.eE.cpct}, there exists $E'$ such that $T_{\bar s}$ 
can be $\e'$-$E'$-approximated, for all $\bar s\in \cY_I$. We may assume $E\subset E'$.  
For every $\bar\lambda\in \D^J$ the operator $T_{\bar\lambda}$ can be written 
as a sum of an operator indexed in $\cY_I$ and one indexed in $\cZ_I$. 
By Lemma~\ref{L.eEapprox} 
it  can be $2\e'$-$E'$-approximated, and therefore \eqref{*.eE} fails for $\e=2\e'$. As $\eps'$ 
is arbitrary, this contradicts~\eqref{*.eE}.

Fix $\eps=\eps'/2$, where $\eps'$ is given by \eqref{**.eE}.

\begin{claim}
For each $E\in\cE$, the subset 
\[U_E=\Big\{\bar{\lambda}\in \D^J\mid T_{\bar{\lambda}}\text{ is }\eps\text{-}E\text{-approximated}\Big\}\]
is closed and has empty interior. 
\end{claim}

\begin{proof}
Suppose  that  $U_E$ is not closed and pick $\bar\lambda\in U^\complement_E\setminus \text{Int}(U^\complement_E)$. Since $T_{\bar\lambda}$ is not $\eps$-$E$-approximated, by Lemma \ref{L.eE2}, there exists a finite rank projection $P\in \ell_\infty(X)$ such that $T_{\bar\lambda}P$ is not $\eps$-$E$-approximated. Fix $\delta>0$. Since $\sum_{j\in J} \theta_j T_j$ strongly converges to a bounded linear operator for every $\bar\theta\in \D^J$, there exists a finite  $I\subset J$  such that  $\|T_{\bar\theta} P\|<\delta$, for all $\bar\theta\in \cZ_I$. Let $\bar\lambda=\bar\lambda_I+\bar\lambda_\infty$, for some $\bar\lambda_I\in \cY_I$ and $\bar\lambda_\infty\in \cZ_I$. As $\bar\lambda\not\in\text{Int}(U^\complement_E)$, there exists $\bar\theta_I\in \cY_I$ and $\bar\theta_\infty\in \cZ_I$ such that $\|T_{\bar\lambda_I}-T_{\bar\theta_I}\|\leq \delta$ and such that  $T_{\bar\theta_I}+T_{\bar\theta_\infty}$ is $\eps$-$E$-approximated. Since
\[T_{\bar\lambda}P=(T_{\bar\theta_I}+T_{\bar\theta_\infty})P+T_{\bar\lambda_I}P-T_{\bar\theta_I}P-T_{\bar\theta_\infty}P+T_{\bar\lambda_\infty}P,\]
Lemma \ref{L.eEapprox} and Lemma \ref{L.eEapprox2}(i) imply that  $T_{\bar\lambda}P$ is $(\eps+3\delta)$-$E$-approximated. As $\delta$ is arbitrary,  Lemma \ref{L.eEapprox2}(ii) implies that $T_{\bar\lambda}P$ is $\eps$-$E$-approximated; contradiction. 

In order to notice that $U_E$ has empty interior, let $\bar\lambda\in \D^J$ and fix a finite $I\subset J$. Let $\bar\lambda_I$ be defined as in the previous paragraph and pick $E'\in\cE$ such that $T_{\bar\lambda_I}$ is $\eps$-$E'$-approximated. Without loss of generality, $E\subset E'$.	 By our choice of $\eps$, there exists $\bar\theta\in \cZ_I$ such that $T_{\bar\theta}$ cannot be $2\eps$-$E'$-approximated. Hence, by Lemma \ref{L.eEapprox}, $T_{\bar\lambda_I}+T_{\bar\theta}$ is not $\eps$-$E'$-approximated. Since $E\subset E'$, this implies $\bar\lambda_I+\bar\theta\not\in U_E$.
\end{proof}

Let $\cU$ be a set of generators of $\cE$ of cardinality $\max\{\aleph_0,|\cE|_{\text{min}}\}$. 
Without loss of generality,  assume that  every element of $\cE$ is contained 
in some element of $\cU$. It follows that
\begin{equation}\label{EqDDD}
\D^J=\bigcup_{E\in\cU}U_E.
\end{equation}
Since   $(|\cE|_{\text{min}},|X|)$ is   small and $|J|\leq |X|$, it follows that $\text{cov}(\D^J)> \max\{\aleph_0,|\cE|_{\text{min}}\}$. On the other hand, since $U_E$ is nowhere dense for every $E\in \cU$, 
\eqref{EqDDD} implies that  $\text{cov}(\D^J)\leq \max\{\aleph_0,|\cE|_{\text{min}}\}$;   contradiction. 
\end{proof}

\begin{proof}[Proof of Theorem~\ref{T.Approximation}]
Fix $\e>0$ and  $F\in \cF$.  
  By Lemma~\ref{LemmaSpakulaWillett},  $\Phi$ is implemented by a unitary operator 
 and therefore continuous with respect to the strong operator topology.  
Let $(T_j)_{j\in J}$ 
be the family of all   $\Phi(e_{yy'})$, for $(y,y')\in F$.  Since $(X,\cE)$  is small cardinal,   Lemma~\ref{L.Approx} implies that  
there exists $E\in \cE$ such that 
$T_{\bar\lambda}$ can be $\e$-$E$-approximated for all $\bar\lambda\in\D^J$, as required.  
\end{proof}

It would be desirable to have $E$ as in the conclusion of Theorem~\ref{T.Approximation} 
depend on~$F$ only, instead of both $F$ and $\e$. Alas, in general this is not true;  see Example~\ref{Ex.Approximate}.

Lemma \ref{LemmaCoarseSpakulaWillett} below
  plays the role of Lemma 4.5 of \cite{SpakulaWillett2013} in our proof. This will be used later to show that the maps we  obtain between $(X,\cE)$ and $(Y,\cF)$ are coarse.
Before stating and proving it, we prove a technical result which will be essential throughout this paper.

\begin{lemma}\label{LemmaNetsCoarseStructure}
Suppose $(X,\cE)$ is a uniformly locally finite coarse space and  $E\in \cE$. 
Let $\cF$ be a directed set and let $(x^F_1)_{F\in \cF}$ and $(x^F_2)_{F\in \cF}$ be nets in $X$ such that $(x_1^F,x^F_2)\in E$, for all $F\in \cF$. Then there exist subsets $I$ and $J$ of $\cF$ and $\varphi:I\to J$ such that $I$ is cofinal 
and the following holds: 
 \begin{enumerate}[(i)]
\item $x_1^F\neq x_1^{F'}$ and $x_2^F\neq x_2^{F'}$, for all distinct $F$ and $F'$ in $J$, and 
\item $x_1^F= x_1^{\varphi(F)}$ and $x_2^F= x_2^{\varphi(F)}$, for all $F\in I$.
\end{enumerate} 
\end{lemma} 

\begin{proof}
By Proposition \ref{partition}, there exists a partition 
\[
X=X_1\sqcup\ldots\sqcup X_k,
\]
such that $(x_1,x_2)\not\in E^{(3)}$, for all $i\leq k$ and all distinct $x_1,x_2\in X_i$. For each $i\leq k$, let 
\[
I_i=\{F\in \mathcal{F}\mid x^F_1\in X_i\}.
\]
Since $\cF$ is directed, there exists  $i\leq k$
 such that $I_i$ is cofinal in $\mathcal{F}$. 
Let $I=I_i$. 
Fix  $F$ and $F'$ in $I$ such that      
 $x_1^F\neq x_1^{F'}$. Then     $(x_1^{F}, x_1^{F'})\not\in E^{(3)}$ and  
since $(x_1^G,x_2^G)\in E$ for all $G\in J$, we conclude that
  $(x_2^{F}, x_2^{F'})\not\in E$, and in particular $x_2^F\neq x_2^{F'}$. 

By an analogous argument and going to a cofinal subset of $I$ if necessary, 
we may assume that $x_2^F\neq x_2^{F'}$ implies  
$x_1^F\neq x_1^{F'}$ for all $F$ and $F'$ in $I$. 

 To recap, we may assume the following:  
\begin{equation}\label{equalifA}
x_1^{F}= x_1^{F'}\ \ \text{if and only if} \ \ x_2^F= x^{F'}_2, \ \ \text{for all}\ \  F,F'\in I.
\end{equation}
Let $\tilde{X}=\{x^F_1\mid F\in I\}$. Fix an assignment  $x\in \tilde{X}\mapsto F_x\in I$ such that $x^{F_x}_1=x$, for all $x\in \tilde{X}$.  Let $J=\{F_x\in I\mid x\in \tilde{X}\}$ be the image of this assignment. Notice that $x^F_1\neq x^{F'}_1$, for all distinct $F, F'\in J$. So, by the backwards implication of \ref{equalifA}, we have that $x^F_2\neq x^{F'}_2$, for all distinct $F, F'\in J$. 

Define $\varphi:I\to J$ by letting $\varphi(F)=F_{x^F_1}$, for all $F\in I$.  By the definition of $\varphi$ and the forward implication in \ref{equalifA},  we have that 
\begin{equation*}
x^F_1=x^{\varphi(F)}_1\ \ \text{and}\ \ x^F_2=x^{\varphi(F)}_2,\ \ 
\text{for all }F\in I.
\end{equation*}
This completes the proof of the lemma.
\end{proof}

If the spaces $(X,\cE)$ and $(Y,\cF)$ are metrizable, the assumption on $\cE$ and $\cF$  in the following lemma follows from 
the Baire category theorem. 

\begin{lemma}\label{LemmaCoarseSpakulaWillett}
Suppose that $(X,\cE)$ and  $(Y,\cF)$ are small 
and uniformly locally finite coarse spaces. 
Let $U:\ell_2(X)\to \ell_2(Y)$ be a unitary operator which spatially implements a    isomorphism $\cstu(X)\to \cstu(Y)$.  
 Then for all $E\in\mathcal{E}$ and all $\delta>0$ the following set belongs to $\cF$:  
\[
F_{E,\delta}\coloneqq \{(y_1, y_2)\in Y^2\mid  
(\exists (x_1,x_2)\in E) (|\langle U\delta_{x_1},\delta_{y_1}\rangle|\geq \delta\wedge|\langle U\delta_{x_2},\delta_{y_2}\rangle|\geq \delta)\}. 
\]
\end{lemma}

\begin{proof}
Suppose otherwise. Then there exist $E\in\mathcal{E}$ and $\delta>0$ 
such that $F_{E,\delta}\notin \cF$. Therefore for every 
 $F\in\mathcal{F}$ there exist $(x^F_1,x^F_2)\in E$ and $(y^F_1,y^F_2)\in Y^2$ such that  
$|\langle U\delta_{x^F_1},\delta_{y^F_1}\rangle|\geq \delta$, $|\langle U\delta_{x^F_2},\delta_{y^F_2}\rangle|\geq \delta$,
 and  $(y^F_1,y^F_2)\not\in F$. 
 
Order   $\mathcal{F}$ by the inclusion; it is a directed set.  
By Lemma \ref{LemmaNetsCoarseStructure}, there exist a cofinal $I\subset \mathcal{F}$, 
   $J\subset \mathcal{F}$,  and a map $\varphi:I\to J$ such that
 \begin{enumerate}[(i)]
\item\label{i.LCSW}  $x_1^F\neq x_1^{F'}$ and $x_2^F\neq x_2^{F'}$, for all distinct $F,F'\in J$, and 
\item\label{ii.LCSW}  $x_1^F= x_1^{\varphi(F)}$ and $x_2^F= x_2^{\varphi(F)}$, for all $F\in I$.
\end{enumerate} 
Fix  $\bar\lambda\in \D^J$. 
Since $(x_1^F,x_2^F)\in E$ for all $F\in I$ and   $\sup_{x\in X} \max(|E_x|, |E^x|)$ 
is finite, 
 \eqref{i.LCSW} implies that 
the sum
$\sum_{F\in J} \lambda_F e_{x^F_1x^F_2}$ 
   converges in the strong operator topology to an operator in $B(\ell_2(X))$. 
Since its support is included in $E$, 
 this operator belongs to 
    $\cstu(X)$.   
   
For each $F\in\cF$, let $e(F)=e_{x^F_1x^F_2}$. Hence, as $\Phi$ is continuous in the strong operator topology, the sum 
\[
\sum_{F\in J}\lambda_F\Phi( e(F))
\]
converges strongly to an operator in $\cstu(Y)$, for every  $\bar\lambda\in \D^J$. 
By Theorem \ref{T.Approximation}, there exists $F_1\in \mathcal{F}$ such that 
\[
\|\chi_A\Phi( e(F))\chi_B\|<\delta^2,
\]
for all $F\in J$ and all $A,B\subset Y$ which are $F_1$-separated. Since $I$ is cofinal in $\mathcal{F}$, we 
can pick $F_1\in I$. For now on, set 
\[
a(x)=x_1^{F_1}, \  b(x)=x_2^{F_1},\  a(y)=y_1^{F_1},\text{ and }b(y)=y_2^{F_1}
\]
 and notice that $\{a(y)\}$ and $\{b(y)\}$ are $F_1$-separated. By (ii), it follows that
\[
\|\chi_{\{a(y)\}}\Phi(e(F_1))\chi_{\{b(y)\}}\|=\|\chi_{\{a(y)\}}\Phi(e(f(F_1)))\chi_{\{b(y)\}}\|<\delta^2.
\]
Using $ \langle U^*\delta_{b(y)},\delta_{b(x)}\rangle\delta_{a(x)}=e(F_1)U^*\delta_{b(y)}$, we have the following
\begin{align*}
\delta^2& \leq  |\langle \delta_{a(y)},U\delta_{a(x)}\rangle \langle U\delta_{b(x)},\delta_{b(y)}\rangle|\\
&=  |\langle \langle U^*\delta_{b(y)},\delta_{b(x)}\rangle\delta_{a(x)},U^*\delta_{a(y)}\rangle|\\
&=|\langle Ue(F_1)U^*\delta_{b(y)},\delta_{a(y)}\rangle|\\
&=\|\chi_{\{a(y)\}}\Phi(e(F_1))\chi_{\{b(y)\}}\|; 
\end{align*}
 contradiction. 
\end{proof}

We are ready to prove \eqref{ThmRigIsoImpliesCoarseEqui} and \eqref{ThmRigIsoImpliesCoarseEquiMARTIN} from the introduction. 

\begin{theorem} \label{T.ThmRigIsoImpliesCoarseEqui}
 Suppose that $(X,\cE)$ and $(Y,\cF)$ are  uniformly locally finite coarse spaces which are also small. 
If  $ \cstu(X)$ and   $\cstu(Y)$ are rigidly    isomorphic, 
then $X$ and $Y$ are coarsely equivalent. 
\end{theorem} 

\begin{proof}
Let $U:\ell_2(X)\to \ell_2(Y)$ be a unitary operator which spatially implements a rigid    isomorphism between $\cstu(X)$ and $\cstu(Y)$. Therefore, there exist  $\delta>0$, $f:X\to Y$ and $g:Y\to X$ such that \[|\langle U\delta_x,\delta_{f(x)}\rangle|\geq \delta\ \ \text{and}\ \  |\langle U^*\delta_y,\delta_{g(y)}\rangle|\geq \delta,\]
for all $x\in X$ and all $y\in Y$.  Lemma \ref{LemmaCoarseSpakulaWillett} implies 
that $f$ and $g$ are coarse maps. Therefore, we only need to verify
 that $g\circ f$ and $f\circ g$  are close to $\Id_X$ and $\Id_Y$, respectively. By our choice of $f$ and $g$, it follows that
\[|\langle U\delta_{g(y)},\delta_{f(g(y))}\rangle|\geq \delta\ \ \text{and}\ \
 |\langle U\delta_{g(y)},\delta_{y}\rangle|=|\langle \delta_{g(y)},U^*\delta_{y}\rangle|\geq \delta,\]
for all  $y\in Y$. Let $F=F_{\Delta_X, \delta}$ be as in the conclusion of 
 Lemma \ref{LemmaCoarseSpakulaWillett}. Then, since $(g(y),g(y))\in \Delta_X$, we have that 
\[(y,f(g(y)))\in F,\]
for all $y\in Y$. This shows that $f\circ g$ is close to $\Id_Y$. Similar arguments show that $g\circ f$ is close to $\Id_X$.  
\end{proof}

We end this section with a metamathematical remark.  

\begin{cor}
($\text{MA}_\kappa$) Suppose that $(X,\cE)$ and $(Y,\cF)$ are $\kappa$-generated uniformly locally finite coarse spaces. 
If  $ \cstu(X)$ and   $\cstu(Y)$ are rigidly    isomorphic, 
then $X$ and $Y$ are coarsely equivalent. \qed
\end{cor}

\begin{remark}
Since our proof of \eqref{ThmRigIsoImpliesCoarseEquiMARTIN} uses Baire category methods, it 
is  amenable to an extension along the lines of using mild forcing axioms from set theory. This is discussed in the appendix of the arXiv version of the present paper (\cite{BragaFarah2018}).
\end{remark}

\section{Coarse spaces with a small partition}\label{SubsectionUnformRoeAlgDisjointUnion} 
In this section, we give a large class of examples of non-metrizable uniformly locally finite coarse spaces such that a rigid    isomorphism 
between their uniform Roe algebras implies coarse equivalence (i.e., \eqref{I.IsoRig} implies \eqref{I.I}). 
 We prove this by studying coarse spaces with a partition that `behave like a metric space'. As before, the reader not interested in set theory  may  replace the words `small' in the definition below by `metrizable'. 

\begin{defi}\label{DefinitionUniformlyMetrizable}
Let  $(X,\cE)$ be a coarse space. We say that $(X,\cE)$ has a \emph{small partition}
if there exists a partition $X=\bigsqcup_{i\in I}X_i$ such that, letting
\[
\cE_\Delta\coloneqq \left\{\bigsqcup_{i\in I}(X_i\times X_i)\cap E\mid E\in\cE\right\},
\]
the coarse space $(X,\cE_{\Delta})$ is small 	 and $\cE$ is generated by $\cE_\Delta\cup\{E\in \cE\mid |E|<\infty\}$.
\end{defi}

Notice that if $(X,\cE)$ is small, then $(X,\cE)$ trivially admits a small partition. We refer the reader to Example \ref{ExUniformlyMetrizableSpace} for more interesting examples of uniformly locally finite coarse spaces which have a  small partition.

We need some simple  lemmas (which are also used in 
Section \ref{SectionRigidityHilbert}). Given a partition $X=\bigsqcup_{i\in I}X_i$ we identify  the product $\prod_{i\in I} B(\ell_2(X_i))$ with a subalgebra of $B(\ell_2(X))$. 

\begin{lemma}\label{LemmaRestrictionToDiag1}
Consider a partition $X=\bigsqcup_{i\in I}X_i$. If $Q\in\prod_{i\in I} B(\ell_2(X_j))$ and  $R\in B(\ell_2(X))$ are
 such that $\chi_{X_i}R\chi_{X_i}=0$, for all $i\in I$  then  $\|Q\|\leq \|Q+R	\|$. 
\end{lemma}

\begin{proof}
For each $i\in I$, identify $\ell_2(X_i)$ with a subspace of $\ell_2(X)$ in the natural way. Then 
($B_{\ell_2(X_i)}$ denotes the unit closed ball of $\ell_2(X_i)$, for  $i\in I$)
\[
\|Q\|=\sup_{i\in I}\sup_{x\in B_{\ell_2(X_i)}}\|\chi_{X_i}Q\chi_{X_i}(x)\|= \sup_{i\in I}\sup_{x\in B_{\ell_2(X_i)}}\|\chi_{X_i}(Q+R)\chi_{X_i}(x)\|
\]
and we conclude that  $\|Q\|\leq \|Q+R\|$.
\end{proof}

\begin{lemma}\label{LemmaRestrictionToDiag}
Let $(X,\cE)$ be a uniformly locally finite coarse space and $X'\subset X$. Fix a partition $X'=\bigsqcup_{i\in I}X_i$ and consider the coarse structure
\[
\cE'=\left\{\bigsqcup_{j\in J}(X_j\times X_j)\cap E\mid E\in \cE\right\}
\]
on $X'$. Then 
$ \cstu(X)\cap \prod_{i\in I} B(\ell_2(X_i))= \cstu(X',\cE')$. 
\end{lemma}

\begin{proof} The inclusion $ \cstu(X)\cap \prod_{i\in I} B(\ell_2(X_i))\supseteq \cstu(X',\cE')$ is clear. 
For the other inclusion, 
let~$Q$ be an operator in   
 $ \cstu(X)\cap \prod_{i\in I} B(\ell_2(X_i))$. To prove    $Q\in \cstu(X',\cE')$,  
  pick sequences $(Q_n)_{n\in\N}$ and $(R_n)_{n\in\N}$ in $\cstu[X]$ such that  
\begin{enumerate}[(i)]
\item $Q_n\in \cstu[X',\cE']$, for all $n\in\N$,
\item $\chi_{X_i}R_n\chi_{X_i}=0$, for all $i\in I$ and all $n\in\N$, and 
\item $\lim_n(Q_n+R_n)=Q$.
\end{enumerate}
By Lemma \ref{LemmaRestrictionToDiag1}, we have that 
\[\|Q-Q_n\|\leq \|Q-(Q_n+R_n)\|.\]
So, $\lim_nQ_n=Q$ and we conclude that $Q\in \cstu(X',\cE')$. 
\end{proof}

A version of Lemma~\ref{L.Approx} holds for coarse spaces admitting  small partitions.

\begin{lemma} \label{T.ApproximationPARTITIONS} 
Suppose that  $(X,\cE)$ is a 
uniformly locally finite coarse space admitting a small partition. Suppose that $(T_j)_{j\in J}$ is a family of finite rank operators in $\cstu(X)$ such that 
$ \sum_{j\in J}\lambda_jT_j$ converges strongly to an operator  $T_{\bar\lambda}\in\cstu(X)$, for every $\bar\lambda\in\D^J$. Then, for every $\eps>0$, there exists $E\in\cE$ such that
\[\|\chi_{\{x\}}T_j\chi_{\{x'\}}\|< \eps,\]
for all $j\in J$ and all $x,x'\in X$ with $(x,x')\not\in E$.  
\end{lemma} 
  
\begin{proof}  
Let $X=\bigsqcup_{i\in I}X_i$ be a small partition of $(X,\cE)$ and let $\cE_\Delta$ be the coarse structure  in Definition \ref{DefinitionUniformlyMetrizable} related to this given partition of $X$. Let $P_j$ denote the projection of $\ell_2(X)$ onto $\ell_2(X_j)$, 
let $\sfP:B(\ell_2(X))\to \prod_{j\in J}B(\ell_2(X_j))$ be the conditional expectation, 
\[
\sfP(T)=\sum_{j\in J} P_j T P_j, 
\]
 and set $\sfR=1-\sfP$. Lemma \ref{LemmaRestrictionToDiag1} implies that $\sfP(\cstu(X))\subset \cstu(X)$.  Since $\sum_{j\in J}\lambda_jT_j$ converges strongly in $\cstu(X)$ for all $\bar\lambda\in \D^J$, so do  $\sum_{j\in J}\lambda_j\sfP T_j$ and $\sum_{j\in J}\lambda_j\sfR T_j$. Clearly, $\sfP T_{\bar{\lambda}}=\sum_{j\in J}\lambda_j\sfP T_j$ and $\sfR T_{\bar{\lambda}}=\sum_{j\in J}\lambda_j\sfR T_j$, for all $\bar\lambda\in\D^J$.

\begin{claim}\label{ClaimKappaGenPar}
For all $\eps>0$, there exists $E\in \cE$ such that $\|\chi_{\{x\}}\sfR T_j\chi_{\{x'\}}\|<\eps$, for all $j\in J$ and all $x,x'\in X$ with  $(x,x')\not\in E$.
\end{claim}

\begin{proof}
Suppose not. Then there exists $\eps>0$, $(j^E)_{E\in \cE}$ in $J$ and $(x_1^E,x_2^E)_{E\in \cE}$ in $X\times X$ with $(x_1^E,x^E_2)\not\in E$ and $\|\chi_{\{x_1^E\}}\sfR T_{j^E}\chi_{\{x_2^F\}}\|\geq \eps$, for all $E\in \cE$. In particular, $\{(x_1^E,x_2^E)\}\in\cE$ and  $(x_1^E,x_2^E)\not\in \bigsqcup_{i\in I}X_i\times X_i$, for all $E\in \cE$. An easy induction produces a  sequence $(j^n)_{n\in\N}$ in $J$ and a sequence $(x^n_1,x^n_2)_{n\in\N}$ in $X\times X$  of distinct elements such that $\|\chi_{\{x_1^n\}}\sfR T_{j^n}\chi_{\{x_2^n\}}\|\geq \eps$ and $(x_1^n,x_2^n)\not\in \bigsqcup_{i\in I}X_i\times X_i$, for all $n\in\N$. Since $T_j$ has finite rank for all $j\in J$, without loss of generality, assume that $(j^n)_{n\in\N}$ is a sequence of distinct elements of $J$. 

Going to a subsequence of $(j^n)_{n\in\N}$ if necessary, we can pick a sequence $(\lambda_n)_{n\in\N}$ in $\{-1,1\}$ such that 
\begin{enumerate}[(i)]
\item $\Big\|\chi_{\{x^n_1\}}\Big(\sum_{i=1}^n \lambda_i\sfR T_{j^i}\Big)\chi_{\{x^n_2\}}\Big\|\geq \eps$, for all $n\in\N$, and
\item $\sum_{i>n}\|\chi_{\{x^n_1\}} \sfR T_{j^i}\chi_{\{x^n_2\}}\|<\eps/2$, for all $n\in\N$.
\end{enumerate}
Therefore, 
\[\Big\|\chi_{\{x^n_1\}}\Big(\sum_{i=1}^\infty \lambda_i \sfR T_{j^i}\Big)\chi_{\{x^n_2\}}\Big\|\geq \frac{\eps}{2},\]
for all $n\in\N$. Since $\cE$ is generated by $\cE_\Delta\cup \{E\in\cE\mid |E|<\infty\}$, there exists $S_1,S_2\in \cstu(X)$ with $\supp(S_1)\in \cE_\Delta$ and $\supp(S_2)$ finite such that 
\[\Big\|\sum_{i=1}^\infty \lambda_i\sfR T_{j^i}-(S_1+S_2)\Big\|< \frac{\eps}{2}.\]
Since $\supp(S_2)$ is finite, fix $n\in \N$ such that $(x^n_1,x^n_2)\not\in \supp(S_2)$. Then
\[\Big\|\chi_{\{x^n_1\}}\Big(\sum_{i=1}^\infty \lambda_i \sfR T_{j^i}-(S_1+S_2)\Big)\chi_{\{x^n_2\}}\Big\|=\Big\|\chi_{\{x^n_1\}}\Big(\sum_{i=1}^\infty \lambda_i \sfR T_{j^i}\Big)\chi_{\{x^n_2\}}\Big\|\geq \frac{\eps}{2};\]
contradiction.
\end{proof}

Fix $\eps>0$ and let $E\in\cE$ be given by Claim \ref{ClaimKappaGenPar} for $\eps/2$. By Lemma \ref{LemmaRestrictionToDiag}, 
$\sfP(\cstu(X))\subset \cstu(X,\cE_\Delta)$. Hence, by Lemma \ref{L.Approx}, going to a larger $E\in \cE$ if necessary, assume that $\sfP T_{j}$ is $\eps/2$-$E$-approximable, for all $j\in J$. This implies  that 
\[\|\chi_{\{x\}}T_j\chi_{\{x'\}}\|\leq \|\chi_{\{x\}}\sfP T_j\chi_{\{x'\}}\|+\|\chi_{\{x\}}\sfR T_j\chi_{\{x'\}}\|\leq \eps,\]
for all $j\in J$ and all $x,x'\in X$ with $(x,x')\not\in E$.
\end{proof}

\begin{thm}\label{ThmPartitionUniformlyMetrizable}
 Suppose  $(X,\cE)$ and $(Y,\cF)$ are  uniformly  locally finite coarse spaces  admitting small partitions.   If $\cstu(X)$ and $\cstu(Y)$ are rigidly    isomorphic, then $(X,\cE)$ and $(Y,\cF)$ are coarsely equivalent.
\end{thm}

\begin{proof}
Fix $\delta>0$, $f:X\to Y$ and $g:Y\to X$ such that \[|\langle U\delta_x,\delta_{f(x)}\rangle|\geq \delta\ \ \text{and}\ \  |\langle U^*\delta_y,\delta_{g(y)}\rangle|\geq \delta,\]
for all $x\in X$ and all $y\in Y$.  Using Lemma~\ref{T.ApproximationPARTITIONS},  one can easily check that Lemma~\ref{LemmaCoarseSpakulaWillett} applies to   $(X,\cE)$ and $(Y,\cF)$. So both $f$ and $g$ are coarse maps. Proceeding   analogously to the proof of Theorem \ref{T.ThmRigIsoImpliesCoarseEqui}, we obtain that $g\circ f$ and $f\circ g$ are close to $\Id_{X}$ and $\Id_{Y}$, respectively. 
\end{proof}

\begin{example}\label{ExUniformlyMetrizableSpace}
We now give a natural way to construct examples of uniformly locally finite coarse spaces with property A which are non-metrizable but have a small partition. In particular, those spaces satisfy the conditions in Theorem  \ref{ThmPartitionUniformlyMetrizable}.

Let  $(X,\mathcal{E})$ be a metrizable coarse space, $J$ be an index set,  $X_j=X$, and $\cE_j=\cE$, for all $j\in J$. Let $\mathfrak{X}=\bigsqcup_{j\in J}X_j$, and for each $j\in J$ identify $X_j$ with a subset of $\fX$ and $E\in\mathcal{E}_j$ with a subset of $\fX\times \fX$ in the natural way. Let $(E_n)_{n\in\N}$ be a sequence generating $\cE$ and for each $j\in J$ let  $(E_{n,j})_{n\in\N}$ be $(E_n)_{n\in\N}$ thought as a sequence of subsets of $X_j\times X_j$. 
Define 

\begin{enumerate}[(i)]
\item  $\fE_1=\{\bigsqcup_{j\in J}E_{n,j}\mid n\in\N\}$, and
\item  $\fE_2=\{E\subset \fX\times \fX\mid |E|<\infty\}$.
\end{enumerate}
Let $\fE$ be the coarse structure on $\fX$ generated by $\fE_1$ and $\fE_2$. The space $(\fX,\fE)$ is uniformly locally finite and, if $J$ is uncountable, it is non-metrizable. It is clear that $(X,\cE)$ has a small partition. Also, one can easily check that the asymptotic dimension of $(X,\cE)$, i.e., $\text{asdim}(X)$, is the same as the asymptotic dimension of $(\fX,\fE)$ (see \cite{RoeBook}, Chapter~9, for definition of asymptotic dimension). Hence, if $\text{asdim}(X)<\infty$, $(\fX,\fE)$ has property A (by \cite{RoeBook}, Section~11.5). 
\end{example}

\section{Isomorphism between uniform Roe algebras and Cartan masas} \label{SectionGhostlyMasas}

The goal of this section is to prove  
\eqref{ThmIsoImpliesRigIsoGhost} from the introduction  (Theorem~\ref{T.ThmIsoImpliesRigIsoGhost} and Theorem~\ref{C.ThmIsoImpliesRigIsoGhost}). 
Only a definition stands between us and this proof. 
A Cartan masa $A$ in $\cstu(X)$ is \emph{ghostly} if
it contains non-compact ghost projections $Q_j$, for $j\in J$, 
which are orthogonal and $\sum_{j\in J} Q_j$ converges strongly to the identity of $\cstu(X)$. 

The following is a strengthening of Lemma 4.6 of \cite{SpakulaWillett2013}, 
where an analogous result was proved using the stronger property A. 

\begin{theorem} \label{T.ThmIsoImpliesRigIsoGhost}
Suppose that  $(X,\cE)$ and $(Y,\cF)$  are metrizable 
uniformly locally finite coarse spaces 
and there exists an isomorphism $\Phi\colon\cstu(X)\to \cstu(Y)$ which is 
not rigidly implemented. 
Then at least one of the following applies. 
\begin{enumerate}[(i)]
\item $\Phi[\ell_\infty(X)]$ is a  ghostly Cartan masa in $\cstu(Y)$. 
\item $\Phi^{-1}[\ell_\infty(Y)]$ is a  ghostly Cartan masa in $\cstu(X)$. 
\end{enumerate}
\end{theorem} 

Dropping the assumption that the spaces be metric, we obtain a slightly weaker conclusion which 
is still incompatible with property A. 

\begin{theorem} \label{C.ThmIsoImpliesRigIsoGhost}
Suppose that  $(X,\cE)$ and $(Y,\cF)$  
are uniformly locally finite coarse spaces   
and  there exists an isomorphism $\Phi\colon\cstu(X)\to \cstu(Y)$ which is 
not rigidly implemented. 
Then at least one of the following applies. 
\begin{enumerate}[(i)]
\item There is a non-compact projection $P\in \ell_\infty(X)$ such that 
$\Phi[P]$ is a  ghost in~$\cstu(Y)$. 
\item There is a non-compact projection $P\in \ell_\infty(Y)$ such that 
$\Phi^{-1}[P]$ is a  ghost in~$\cstu(X)$. 
\end{enumerate}
\end{theorem} 

Before proceeding to prove these theorems, we point out that we do not know an example of a
uniform Roe algebra with a ghostly Cartan masa, or even an example of a uniform 
Roe algebra with a  Cartan masa that contains a noncompact ghost projection. 
The best that we could do is Example~\ref{Ex.Ghost}.

\begin{proof}[Proof of Theorem \ref{T.ThmIsoImpliesRigIsoGhost}]
 If $\cstu(X)$ and $\cstu(Y)$ are    isomorphic,  Lemma \ref{LemmaSpakulaWillett} gives us a unitary operator $U:\ell_2(X)\to \ell_2(Y)$ which spatially implements an isomorphism between $\cstu(X)$ and $\cstu(Y)$. Without loss of generality, we assume that $U$ does not induce a rigid    isomorphism. Hence, for every $n\geq 1$, the set
\[
X_n\coloneqq \Big\{x\in X\mid \max_{y\in Y} |\langle U\delta_{x},\delta_y\rangle|\geq  2^{-n}\Big\}.
\]
is a proper subset of $X$. 

\begin{claim} \label{C.Zm} There  is a partition $X=\bigcup_m Z_m$ such that each $Z_m$
is infinite but $Z_m\cap (X_{n+1}\setminus X_n)$
has at most one element for all $m$ and all $n$. 
\end{claim} 

\begin{proof} 
Since $\bigcup_{n\in \N} X_n=X$, for every $n\in \N$ the set $X\setminus X_n$ is infinite for all $n$.  
Since $X$ is countable, we can find a partition $X=\bigcup_{m\in \N} Z_m'$ 
such that $Z_m'\cap X_n$ is finite for all $m$ and $n$. 
Every $Z_m'$ can now be partitioned into sets as required. 
\end{proof} 

Fix $m\in\N$.  Let  $(x_n)_{n\in\N}$ be an enumeration of   $Z_m$ as in Claim~\ref{C.Zm} 
such that 
$x_n\in X\setminus X_n$ for all $n$. Then  
 $\|Ue_{x_nx_n}U^*\delta_y\|<2^{-n}$, for all $n\in\N$ and all $y\in Y$. 
 So, $\sum_{n\in\N}e_{x_nx_n}$ converges in the strong operator topology to an operator in $\cstu(X)$. 
 
\begin{claim} \label{C.2} The projection 
$P_m=\sum_{n\in\N}Ue_{x_nx_n}U^*
$
is a ghost.
\end{claim} 

\begin{proof} 
Let $\eps>0$. Pick $n_0\in \N$ such that $2^{-n_0}<\eps/2$. Since $Ue_{x_nx_n}U^*$ is compact (it has rank $1$), for all $n\in\N$, we can pick a finite  $A\subset X$ such that 
\[|\langle Ue_{x_nx_n}U^*\delta_y,\delta_{y'}\rangle|< \frac{\eps}{2n_0},\] 
for all $n\in\{1,\ldots,n_0\}$ and all $y,y'\in Y\setminus A$. Then, for all $y,y'\in Y\setminus A$, we have that
\begin{align*}
|\langle P_m\delta_y,\delta_{y'}\rangle |&\leq \sum_{n=1}^{n_0}|\langle Ue_{x_nx_n}U^*\delta_y,\delta_{y'}\rangle|+\sum_{n> n_0}|\langle Ue_{x_nx_n}U^*\delta_y,\delta_{y'}\rangle|\\
&< \frac{\eps}{2}+\sum_{n> n_0}\| Ue_{x_nx_n}U^*\delta_y\|\\
&< \frac{\eps}{2}+\frac{\eps}{2}=\eps.
\end{align*}
This shows that $P_m$ is a ghost.
\end{proof} 

The projections $P_m$, for $m\in\N$, are orthogonal, non-compact,  ghost projections in  $A\coloneqq \Ad U(\ell_\infty(X))$ and $\sum_{m\in\N}P_m$ is the identity in $\cstu(Y)$. Since $\ell_\infty(X)$ is a Cartan masa in $\cstu(X)$,  $A$ is a Cartan masa in $\cstu(Y)$. Therefore,  $A$ is a ghostly Cartan masa in $\cstu(Y)$. 
\end{proof}

\begin{proof}[Proof of Theorem~\ref{C.ThmIsoImpliesRigIsoGhost}] 
Define sets $X_n$, for $n\in \N$, as in the proof of Theorem~\ref{T.ThmIsoImpliesRigIsoGhost}. 
Then $X\setminus X_n$ is infinite for all $n$, and we can 
choose $Z\subseteq X$ such that $Z\cap X_{n+1}\setminus X_n$
has at most one element for all $n$. 
The proof of Claim~\ref{C.2} 
shows that $P_m=\sum_{n\in\N}Ue_{x_nx_n}U^*$ is a ghost projection.  
It belongs to the Cartan masa $U \ell_\infty(X) U^*$ and it is clearly not compact. 
\end{proof} 

The following is as close as we could get to the conclusion of 
Theorem~\ref{T.ThmIsoImpliesRigIsoGhost}. 
It also serves as a basis for a limiting example for possible improvements of Theorem~\ref{T.Approximation} 
(Example~\ref{Ex.Approximate}) promised earlier. 

\begin{example} \label{Ex.Ghost} 
There exist
uniformly locally finite metric spaces $(X,d)$ and $(Y,\partial)$  such that  
$\cstu(Y)$ is    isomorphic to a corner $P\cstu(X)P$ where $P$
is a noncompact ghost projection in $\cstu(X)$. 

The metric space $(X,d)$ is the counterexample to the coarse Baum--Connes conjecture given 
in \cite{higson2002counterexamples} on pages 348--349, a copy of which we assume that the reader has handy. 
Therefore $X=\bigsqcup_{n\in \N} X_n$, each $X_n$ is  finite,  and 
the distance between $x\in X_n$ and $x'\in X_m$ is (for definiteness) $m+n$ if $m\neq n$. 
Each $X_n$ is the set of vertices of a finite graph $G_n$ with the shortest path distance. 
The graphs  $G_n$ were chosen to be   a sequence of $k$-$\e$-expander
graphs for a fixed $k$ and $\e>0$.  
Let $\Delta_n$ be the Laplacian of $G_n$; this is the operator on $\ell_2(X_n)$
given by the matrix 
\[
kI-A, 
\]  
where $A$ is the incidence matrix of $G_n$ and $I$ is the identity. Then $\Delta_n$ is positive and  $\|\Delta_n\|\leq k$ for all $n$. 
Therefore the direct sum $\Delta$ of all $\Delta_n$ is positive and 
belongs to $B(\ell_2(X))$. By the definition,  
$\supp(\Delta)\subseteq \{(x,x'): d(x,x')\leq 1\}$ and therefore $\Delta\in \cstu(X)$. 

Let $P$ denote the projection to $\ker(\Delta)$. 
Since the lowest nonzero eigenvalue of~$\Delta$ is bounded away from zero, 
the restriction of the operator $T\coloneqq 1-\Delta\|\Delta\|^{-1}$ to $\ker(\Delta)^\perp$ has  norm strictly less than 1.  
Therefore the sequence $(T^n)_{n\in\N}$ converges to $P$ 
\emph{in norm}, and $P\in \cstu(X)$.
The range of $P\chi_{X_n}$ consists of constant functions in $\ell_2(X_n)$. 
Thus (writing $m_n=|X_n|$) $P\chi_{X_n}$ is given by the $m_n\times m_n$ matrix 
all of whose entries are equal to $m_n^{-1/2}$. 
Since $|m_n|\to \infty$, $P$ is a ghost.   

 Let $\xi_n$ be a unit vector in the range of $P\chi_{X_n}$. 
Let $Y$ be the metric space with domain $\N$ and the metric  $\partial (m,n)=m+n$. 
Then $U\colon \ell_2(Y)\to \ell_2(X)$ defined by $U(\delta_m)=\xi_m$ 
implements a    isomorphism between $\cstu(Y)$ and $P\cstu(X)P$. 
\end{example} 

The following example shows that  Theorem~\ref{T.Approximation} 
cannot be improved by stating that for $\Phi\colon \cstu(Y)\to \cstu(X)$ the uniformity 
$E\in \cE$ depends only on $F\in \cF$ even when $E=\Delta_Y$. 

\begin{example} \label{Ex.Approximate} There 
exist metrizable uniformly locally finite coarse spaces $(X,\cE)$ and $(Y,\cF)$,   
a     $*$-homomorphism $\Phi\colon \cstu(Y)\to \cstu(X)$, such that  
for every $E\in \cE$ there exist $\e>0$ 
for which  $\Phi(1)$ cannot be $\e$-$E$-approximated. 

Let $(X,\cE)$,  $(Y,\cF)$, $P$ and $\Phi:\cstu(Y)\to P\cstu(X) P$  be  as in Example~\ref{Ex.Ghost}.  
We have that  $\Phi(1)$ is a ghost projection in $\cstu(X)$. Given $E\in \cE$. The fact that the    $*$-homomorphism $\Phi$ is not unital is easily remedied: add a single point $y$ to 
 $Y$ and send $e_{yy}$ to $1-P$ (the complement of the ghost projection 
in $\cstu(X)$). 
\end{example} 

We do not know whether if  the map $\Phi$ as in Theorem~\ref{T.Approximation} and Example~\ref{Ex.Approximate} 
is assumed to be an    isomorphism one can deduce that it has a stronger coarse-like property.

\begin{cor}\label{CorPartitionUniformlyMetrizable}
 Suppose  $(X,\cE)$ and $(Y,\cF)$ are  uniformly  locally finite coarse spaces admitting small  partitions.   Suppose that  all ghost projections in $\cstu(X)$ and $\cstu(Y)$  are compact. If $\cstu(X)$ and $\cstu(Y)$ are    isomorphic, then $(X,\cE)$ and $(Y,\cF)$ are coarsely equivalent. 
\end{cor}

\begin{proof}
This follows straightforwardly from Theorem \ref{C.ThmIsoImpliesRigIsoGhost} and  Corollary \ref{ThmPartitionUniformlyMetrizable}.
\end{proof}

\section{Rigidity for spaces which coarsely embed into Hilbert space} \label{SectionRigidityHilbert}
In order to  prove  \eqref{ThmIsoImpliesRigIso} from the introduction (Theorem~\ref{T.CorHilbertEmbRigidity} below),  
we need to introduce the notion of the Rips complex and coarse connected components. 

\begin{defi}
Let $(X,\cE)$ be a coarse space and let $E\in \cE$ be  symmetric and such that $\Delta_X\subset E$. Let 
\[P_E(X)=\{A\subset X\mid (x,x')\in E,\  \forall x,x'\in A\}.\]
We call $P_E(X)$ the \emph{Rips complex of $(X,\cE)$ over $E$}. We define an equivalence relation $\sim_E$ on $P_E(X)$ by setting $A\sim_E A'$ if there exist $x_1,\ldots,x_n\in X$ such that $x_1\in A$, $x_n\in A'$, and $(x_m,x_{m+1})\in E$, for all $m\in\{1,\ldots,n-1\}$.
\end{defi}

\begin{defi} Let $(X,\cE)$ be a uniformly locally finite coarse space. We say that $(X,\cE)$ has an \emph{infinite coarse component} if there exists a symmetric $E\in \cE$, with $\Delta_X\subset E$, such that the Rips complex $P_E(X)$ has an infinite $\sim_E$-equivalence class. Otherwise, we say that $(X,\cE)$ has \emph{only finite coarse components}.
\end{defi}

Let us present  the prototypical example of a coarse (metric) space with only finite coarse connected components. Let $(X_n)_{n\in \N}$ be a sequence of finite metric spaces  which are uniformly locally finite, uniformly  in the index $n$. We define a metric space $(X,d)$, called the \emph{coarse disjoint union of $(X_n)_n$}, by letting $X=\bigsqcup_{n\in\N}X_n$ and picking a metric $d$ on $X$ which is  given by the metric on each component and such that $d(X_n, X_m ) \to\infty$ as $ n + m \to\infty$.  Any such metric is unique up to coarse equivalence. It is straightforward to check that  $X$ has only finite coarse components.

The following lemma is proved by going to a library.

\begin{lemma}\label{PropFinnSell}
Let $(X,d)$ be a uniformly locally finite metric space with only finite coarse components which coarsely embeds into a Hilbert space. Then all ghost projections in $\cstu(X)$ are compact. \end{lemma}

\begin{proof} 
A metric $d$ is uniformly discrete if $\inf_{x\neq x'}d(x,x')>0$. 
It is clear that every  coarse metric space $(X,d)$ carries a uniformly discrete metric $\partial$ 
compatible with its coarse structure (let $\partial(x,x')=d(x,x')+1$ if $x\neq x'$). 

A uniformly locally finite metric space which coarsely embeds into a Hilbert space must satisfy the coarse Baum--Connes conjecture by  \cite{Yu2000}, Theorem 1.1.
This mplies that every ghost projection in $\cstu(X)$ is compact  
 (this is the second sentence of Proposition 35 of \cite{FinnSell2014}; note that $\mu_0$ should be $\mu$). 
 \end{proof}

We are now ready to prove 
\eqref{ThmIsoImpliesRigIso} from the introduction.

\begin{theorem}\label{T.CorHilbertEmbRigidity}
If 
$(X,\cE)$ and $(Y,\cF)$  
are uniformly locally finite coarse spaces 
such  that 
both $(X,\cE)$ and $(Y,\cF)$ coarsely embed into a Hilbert space
then every isomorphism between $\cstu(X)$ and $\cstu(Y)$ 
is rigidly implemented. 
\end{theorem}

\begin{proof}
Since each of  $(X,\cE)$ and $(Y,\cF)$ coarsely embeds into a Hilbert space (which is in particular a metric space), it follows that both $\cE$ and $\cF$ are countably generated. Hence,  $(X,\cE)$ and $(Y,\cF)$ are metrizable \cite[Theorem 2.55]{RoeBook}. Let $d$ and $\partial $ be metrics on $X$ and $Y$, respectively, such that $\cE_d=\cE$ and $\cE_\partial=\cF$. 

Let $U:\ell_2(X)\to\ell_2(Y)$ be a unitary spatially implementing an    isomorphism between $\cstu(X)$ and $\cstu(Y)$. By symmetry, it suffices to show that there exist $\delta>0$ and $f:X\to Y$ such that $|\langle U\delta_x,\delta_{f(x)}\rangle|\geq \delta$, for all $x\in X$.  Assuming this is not the case, we obtain a sequence $(x_n)_{n\in\N}$ of distinct points in $X$ such that $\|Ue_{x_nx_n}U^*\delta_y\|<2^{-n}$, for all $n\in\N$ and all $y\in Y$. For each $n\in \N$,  set 
\[S_n=Ue_{x_nx_n}U^*.\]
So, $S_n$ is a projection on the $1$-dimensional subspace generated by $\zeta_n\coloneqq U\delta_{x_n}$, i.e., $S_n=\langle \cdot, \zeta_n\rangle \zeta_n$, for all $n\in\N$. Since  $(x_n)_{n\in\N}$
is a sequence of distinct points,  $\sum_{n\in I}S_n$ converges in the strong operator topology to an operator in $\cstu(Y)$, for all $I\subset \N$.

Let us pick a sequence of finite subsets $(Y_k)_{k\in \N}$ of $Y$ and a subsequence $(\zeta_{n_k})_{k\in\N}$ such that

\begin{enumerate}[(i)]
\item $Y_{k}\cap Y_{\ell}=\emptyset$, for all $k\neq \ell$ in $\N$,
\item $\partial(Y_{k},Y_{\ell})\to \infty$, as $k+\ell\to \infty$, and
\item $\|\zeta_{n_k}-\chi_{Y_k}\zeta_{n_k}\|< 2^{-k}$, for all $k\in\N$.
\end{enumerate}

We proceed by induction on $k\in\N$. Let $Y_1\subset Y$ be a finite subset such that $\|\zeta_1-\chi_{Y_1}\zeta_1\|< 2^{-1}$ and set $n_1=1$. Let $k\geq 2$ and assume that $Y_j$ and $n_j$ have been defined, for all $j\leq k-1$. Fix $y_0\in Y$ and pick $r>0$ such that
\[
\partial \Big(\bigcup_{j=1}^{k-1}Y_j,B_r(y_0)^\complement\Big)>k.
\]
Let $Y'=B_r(y_0)^\complement$ (the complement of the $r$-ball centered at $y_0$). Since $(\zeta_n)_{n\in\N}$ is an orthonormal sequence,  $(\zeta_n)_{n\in\N}$ is weakly null.  Therefore, since $B_r(y_0)$ is finite, we can pick $n_{k}\in\N$ such that $\|\zeta_{n_k}-\chi_{Y'}\zeta_{n_k}\|< 2^{-k-1}$. Pick a finite $Y_k\subset Y'$ such that $\|\chi_{Y'}\zeta_{n_k}-\chi_{Y_k}\zeta_{n_k}\|< 2^{-k-1}$. The sequences $(Y_k)_{k\in\N}$ and $(\zeta_{n_k})_{k\in\N}$ have the desired properties.

For each $k\in\N$, let $\xi_k=\chi_{Y_k}\zeta_{n_k}/\|\chi_{Y_k}\zeta_{n_k}\|$ and let  $P_k\in B(\ell_2(Y))$ be the $1$-dimensional  projection on the subspace generated by $\xi_k$. So, $P_k=\langle \cdot, \xi_k\rangle \xi_k$. We clearly have  that 
\[\|S_{n_k}-P_k\|< 2^{-k+2}\ \ \text{and}\ \  P_k=\chi_{Y_k}P_k\chi_{Y_k},\]
for all $k\in\N$. In particular, $P_k\in B(\ell_2(Y_k))$ and $\|P_k\delta_y\|< 2^{-k+3}$, for all $k\in\N$ and all $y\in Y$. 

Since $Y_k$ is finite, $\supp(P_k)$ is finite and $P_k\in\cstu[Y]$. Therefore, $S_{n_k}-P_k\in \cstu(Y)$, for all $k\in\N$. Since 
\[\sum_{k\in \N}\| S_{n_k}-P_k\|<\infty,\]
it follows that $\sum_{k\in \N}(S_{n_k}-P_k)$ converges in norm to an operator in $\cstu(Y)$. Hence, as $\sum_{k\in \N}S_{n_k}\in \cstu(Y)$, we have that $\sum_{k\in \N} P_k\in \cstu(Y)$. 

Let $\tilde{P}=\sum_{k\in\N} P_k$. By  (i) above, and since $P_k\in B(\ell_2(Y_k))$ and $\xi_k\in\ell_2(Y_k)$, we have that $\tilde{P}\xi_k=\xi_k$, for all $k\in\N$. Let $\tilde{Y}=\bigcup_{k\in \N}Y_k$. Therefore, since  
\[\tilde{P}\in \prod_{k\in\N}B(\ell_2(Y_{k})),\] 
 Lemma \ref{LemmaRestrictionToDiag} implies that $\tilde{P}\in \cstu(\tilde{Y})$.

\begin{claim}\label{Claim1} The projection $\tilde{P}$ is a ghost  in $ \cstu(\tilde{Y})$.
\end{claim} 
\begin{proof} 
Since $\|P_k\delta_y\|\leq 2^{-k+3}$, for all $k\in\N$ and all $y\in Y$,  proceeding  as in the proof of 
Theorem~\ref{T.ThmIsoImpliesRigIsoGhost}, we obtain that $\tilde{P}$ is a ghost. Since $P_k$ is a projection, for all $k\in\N$, and since $P_kP_\ell=0$, for all $k\neq \ell$, we have that $\tilde{P}$ is a projection. 
\end{proof} 

\begin{claim}\label{Claim2} All ghost projections in $\cstu(\tilde{Y})$ are compact.
\end{claim} 

\begin{proof} 
Since $Y$ coarsely embeds into a Hilbert space, so does  $\tilde{Y}$.  By our discussion on finite coarse components preceding this lemma, it is clear that  $(\tilde{Y},\partial)$  has only finite coarse components.  Therefore,   Lemma~\ref{PropFinnSell} implies that all ghost projections in $\cstu(\tilde{Y})$ must be compact.
\end{proof} 

Claim~\ref{Claim1} and Claim~\ref{Claim2} together imply 
 that $\tilde{P}$ is  compact, contradicting the facts that $(\xi_k)_{k\in\N}$ is orthonormal and that $\tilde{P}\xi_k=\xi_k$, for all $k\in\N$. 
\end{proof}

\section{Isomorphism between algebraic uniform Roe algebras}\label{SectionIsoAlgebraicRoeAlg}

In this section, we prove rigidity of algebraic uniform Roe algebras for general uniformly locally finite coarse spaces
 (\eqref{ThmMAIN} from the introduction). This is summarized in Theorem \ref{T.ThmMAIN} below and part of its proof is inspired by the methods in \cite{WhiteWillett2017}. We refer the reader to \cite[Proposition 3.10]{ChungLi2018} for a similar result about uniformly locally finite metric spaces.

\begin{thm}\label{T.ThmMAIN}
Let $(X,\mathcal{E})$ and $(Y,\mathcal{F})$ be uniformly locally finite coarse spaces. The following are equivalent.

\begin{enumerate}[(i)]
\item $(X,\mathcal{E})$ and $(Y,\mathcal{F})$ are bijectively coarsely equivalent.
\item $\cstu[X]$ and $\cstu[Y]$ are    isomorphic.
\item There is an    isomorphism $\Phi:\cstu(X)\to \cstu (Y)$ such that $\Phi(\ell_\infty(X))\subset\ell_\infty(Y)$. 
\end{enumerate} 
\end{thm}

Although the implication (i) $\Rightarrow$ (ii) in Theorem \ref{T.ThmMAIN}   is quite straightforward and it  has already been proved for metric spaces with bounded geometry (\cite{WhiteWillett2017}, Corollary 1.16),  for the convenience of the reader and for completeness, we present a proof of this result here.

\begin{proof}[Proof of (i) $\Rightarrow$ (ii) of Theorem \ref{T.ThmMAIN}]
Let $f:(X,\mathcal{E})\to (Y,\mathcal{F})$ be a bijective coarse equivalence. Define an operator $U:\ell_2(X)\to \ell_2(Y)$ by letting  $U\delta_x=\delta_{f(x)}$, for each $x\in X$. Since $f$ is a bijection, it follows that $U$ is a unitary isomorphism. Let $E\in\mathcal{E}$. So there exists $F\in\mathcal{F}$ such that $(x,x')\in E$ implies $(f(x),f(x'))\in F$. Let us show that $\supp(UTU^*)\subset F$, for all $T\in \cstu(X)$ with $\supp(T)\subset E$.

Notice that
\[\langle UTU^*\delta_y,\delta_{ y'}  \rangle=\langle TU^*\delta_y,U^*\delta_{ y'}  \rangle=\langle T\delta_{f^{-1}(y)},\delta_{f^{-1}(y')}  \rangle\neq  0\]
if and only if $(f^{-1}(y),f^{-1}(y'))\in \supp(T)$. Say  $(y,y')\not\in F$. Then, by our choice of $F$,    $(f^{-1}(y),f^{-1}(y'))\not\in E$. Therefore,  as $\supp(T)\subset E$, this implies that $\langle UTU^*\delta_y,\delta_{ y'}  \rangle= 0$. This shows that $\supp(UTU^*)\subset F$. 

Define an    isomorphism $\Phi:B(\ell_2(X))\to B(\ell_2(Y))$ by letting $\Phi(T)=UTU^*$, for all $T\in B(\ell_2(X))$. The discussion above shows that $\Phi(\cstu[X])\subset \cstu[Y]$ and, by symmetry,  it follows that $\Phi^{-1}(\cstu[Y])\subset  \cstu[X]$. Therefore, $\Phi$ is an    isomorphism between $\cstu[X]$ and $\cstu[Y]$. 
\end{proof}

We now turn to the proof of (ii) $\Rightarrow$ (iii) in Theorem \ref{T.ThmMAIN}.  We show that if $\Phi:\cstu[X]\to \cstu[Y]$ is an    isomorphism, then $\Phi$ satisfies the following `coarse-like' property: there exists  an assignment $E\in\mathcal{E}\mapsto F_E\in\mathcal{F}$ such that $\supp(T)\subset E$ implies $\supp(\Phi(T))\subset F_E$, for all $T\in \cstu[X]$ (see Theorem \ref{ThmAlgebraicMain} below).

\begin{lemma}\label{LemmaFiniteSupp}
Let $(X,\mathcal{E})$  be a uniformly locally finite coarse space. Then $\supp(T)$ is finite, for every 
rank $1$ operator $T\in \cstu[X]$. 
\end{lemma}

\begin{proof}
Since $T$ is a  rank $1$ operator, $T=\xi\odot \eta$ for some vectors $\xi,\eta\in \ell_2(X)$. Since $(X,\cE)$ is uniformly locally finite and $\supp(T)\in \cE$, the supports of $\xi$ and $\eta$ must be finite. As   
$\supp(\xi\odot \eta)\subseteq\supp(\xi)\times \supp(\eta)$, the proof is finished. 
\end{proof}

The following technical lemma will be essential in the proof of Theorem \ref{ThmAlgebraicMain}.

\begin{lemma}\label{LemmaAlgebraicMain}
Let $(X,\mathcal{E})$ and $(Y,\mathcal{F})$ be uniformly locally finite coarse spaces and assume that $\Phi:\cstu[X]\to \cstu[Y]$ is a    isomorphism. Let $E\in\mathcal{E}$ and let $(T_j)_{j\in J} $ be a bounded family of rank $1$ operators in $\cstu[X]$  with mutually orthogonal images and such that $\supp(T_j)\subset E$, for all $j\in J$. Then,  \[\bigcup_{j\in J}\supp(\Phi(T_j))\in \mathcal{F}.\] 
\end{lemma}

\begin{proof}
The proof consists of a series of claims. Throughout this proof,  fix $E\in\mathcal{E}$. 

\begin{claim} \label{Claim111}  Suppose $(T_j)_{j\in \N} $ is a bounded family of rank $1$ operators in $\cstu[X]$ such that $\supp(T_j)\subset E$, for all $j\in \N$.  Then, there exists an operator $T\in\cstu[X]$ such that $\supp(T)\subset E$ 
and  
\[
\supp(\Phi(T))=\bigcup_{j\in \N}\supp(\Phi(T_j)).
\]
In particular, 
\[
\bigcup_{j\in \N}\supp(\Phi(T_j))\in \mathcal{F}.
\]  
\end{claim} 

\begin{proof} 
For each $j\in \N$, let 
\[
m_j\coloneqq \inf\Big\{|\langle \Phi(T_j)\delta_y,\delta_{y'}    \rangle| \mid  (y,y')\in \supp(\Phi(T_j))\Big\}.
\]
By Corollary \ref{LemmaRank} and  Lemma \ref{LemmaFiniteSupp},  $\supp(\Phi(T_j))$ is finite, so  $m_j>0$, for all $j\in \N$. Let $M=\sup_{j\in\N}\|T_j\|$. Define a sequence $(\lambda_j)_{j\in \N}$ of positive reals as follows. Let $\lambda_1=1$ and assume that $\lambda_{j-1}$ had been defined for $j\geq 2$. Let $\lambda_j$ be a positive real smaller than \[\frac{\min\{m_1,\ldots, m_{j-1}\}}{ 2^{j}M}.\]
Clearly, $\sum_j|\langle \lambda_jT_j\delta_x,\delta_{x'}\rangle|\leq M$, for all $x,x'\in X$. Also, $\langle \lambda_jT_j\delta_x,\delta_{x'}\rangle=0$, for all $(x,x')\not\in E$. Hence, by the definition of $\cstu[X]$, it follows that $T\coloneqq \sum_{j\in\N}\lambda_jT_j\in \cstu[X]$ is well-defined.  So, $\Phi(T)\in \cstu[Y]$ is defined and $\supp(\Phi(T))\in\mathcal{F}$. At last, notice that the sequence  $(\lambda_j)_j$ was chosen such that 
\[\supp(\Phi(T))=\bigcup_{j\in \N}\supp(\Phi(T_j)).\]
This completes the proof. 
\end{proof} 

\begin{claim} \label{Claim222}   Let  $(T_j)_{j\in J} $ be an infinite family of nonzero operators in $B(\ell_2(X))$  with mutually orthogonal images. There exists an infinite $J'\subset J$ and a family $(x_j,x'_j)_{j\in J'}$ of pairwise distinct elements  in $X\times X$ 
such that $(x_j,x'_j)\in \supp(T_j)$, for all $j\in J'$. 
\end{claim} 
\begin{proof} 
Since the operators $(T_j)_{j\in J} $ have mutually orthogonal images, it is clear that, for all infinite $J_0\subset J$,  there is no finite  $A\subset X\times X$ such that $\supp(T_j)\subset A$, for all $j\in J_0$.  We now construct $J'\subset J$ and the required  family $(x_j,x'_j)_{j\in J'}$ by induction. Pick any $j_1\in J$ and any $(x_{j_1}, x_{j_1}')\in \supp(T_{j_1})$. Let $n\geq 1$ and assume that $j_k$ and $(x_{j_k}, x_{j_k}')$ have already been picked, for all $k\in\{1,\ldots,n\}$. Let \[A=\bigcup_{k=1}^n \{(x_{j_k}, x_{j_k}')\}\ \ \text{and}\ \ J_0=J\setminus \{j_1,\ldots,j_n\}.\]
Since $A$ is finite and $J_0$ is infinite, there exists $j_{k+1}\in J_0$ such that $\supp(T_j)\not\subset A$. Now pick $(x_{j_{k+1}}, x_{j_{k+1}}')\in \supp(T_j)\setminus A$. This proves the claim. 
\end{proof} 

\begin{claim} \label{Claim333}   Let $A\subset X\times X$ be finite and let $(T_j)_{j\in J} $ be an infinite family of rank $1$ operators in $B(\ell_2(X))$  with mutually orthogonal images  such that $ A\cap \supp(T_j)\neq 0$, for all $j\in J$. Then, there exist $x\in X$, an  infinite  $J'\subset J$ and  a family   $(x_j)_{j\in J'}$ in $X$ of pairwise distinct elements such that either 
\begin{enumerate}[(i)]
\item $(x,x_j)\in \supp(T_j)$, for all $j\in J'$, or 
\item $(x_j,x)\in \supp(T_j)$, for all $j\in J'$.
\end{enumerate}
\end{claim} 

\begin{proof} 
Let  $A\subset X\times X$ and  $(T_j)_{j\in J} $ be as above. Since $A$ is finite, a simple pigeonhole argument gives us $x,x'\in X$ and an infinite $J_0\subset J$  such that $(x,x')\in \supp(T_j)$, for all $j\in J_0$. By Claim~\ref{Claim222}, pick  a countably  infinite  $J_1\subset J_0$ and a family $(x_j,x_j')_{j\in J_1}$ of distinct elements such that $(x_j,x'_j)\in \supp(T_j)$, for all $j\in J_1$. Since $J_1$ is infinite, there exists an infinite  $J'\subset J_1$ such that either $(x_j)_{j\in J'}$ or $(x'_j)_{j\in J' }$ is a sequence of pairwise distinct elements. Assume that $(x_j)_{j\in J'}$ is a sequence of pairwise distinct elements. 

As the operators $(T_j)_{j\in J'} $ have rank $1$, so do their adjoints $(T^*_j)_{j\in J'}$. Since $(x,x'),(x_j,x'_j)\in \supp(T_j)$, for all $j\in J'$, we have that  $T^*\delta_{x'}\neq 0$ and $T^*_j\delta_{x'_j}\neq 0$, for all $j\in J' $. Hence, for all $j\in J'$, there exists $\lambda_j\neq 0$ such that $T^*_j\delta_{x'}=\lambda_j T_j^*\delta_{x'_j}$. It follows that 
\[\langle T_j\delta_{x_j},\delta_{x'}\rangle=\langle \delta_{x_j},T^*_j\delta_{x'}\rangle=\langle \delta_{x_j},\lambda_j T_j^*\delta_{x'_j}\rangle=\overline{\lambda}_j\langle T_j\delta_{x_j},\delta_{x'_j}\rangle\neq 0,\]
for all $j\in J' $. So, $(x_j,x')\in\supp(T_j)$, for all $j\in J' $, and (ii) holds. If we had assumed that  $(x'_j)_{j\in J'}$ is a sequence of distinct elements, similar arguments would give us that (i) holds. 
\end{proof}

\begin{claim}\label{Claim444}   Let $A\subset X\times X$ be finite and let $(T_j)_{j\in J} $ be a family of rank $1$ operators in $\cstu[X]$  with mutually orthogonal images and such that $\supp(T_j)\subset E$, for all $j\in J$. Then,  
\[
\Big|\{j\in J\mid A\cap \supp(T_j)\neq \emptyset\}\Big|<\infty.
\] 
\end{claim}

\begin{proof} 
Suppose the claim does not hold and let $x\in X$, $J'\subset J$ and  $(x_j)_{j\in J'}$ be given by Claim~\ref{Claim333}. Without loss of generality,  assume that $(x,x_j)\in \supp(T_j)$, for all $j\in J'$. Hence, as $\supp(T_j)\subset E$, for all $j\in J'$, it follows that 
\[\{x_j\mid j\in J' \}\subset E_x.\]
Since $(x_j)_{j\in J'}$ is an infinite sequence of distinct elements, this shows that $|E_x|=\infty$, which contradicts the fact that $(X,\mathcal{E})$ is uniformly locally finite.
\end{proof} 

\begin{claim}\label{Claim555}   Let  $B\subset Y\times Y$ be finite and let $(T_j)_{j\in J} $ be a bounded family of rank $1$ operators in $\cstu[X]$  with mutually orthogonal images and such that $\supp(T_j)\subset E$, for all $j\in J$. 
Then,  
\[
\Big|\{j\in J\mid B\cap \supp(\Phi(T_j))\neq \emptyset\}\Big|<\infty.
\] 
\end{claim} 

\begin{proof} Assume the claim does not hold. By Corollary \ref{LemmaRank}, $(\Phi(T_j))_{j\in J}$ is a  bounded family of rank $1$ operators in $\cstu[Y]$  with mutually orthogonal images. Let $y\in Y$, $J'\subset J$ and  $(y_j)_{j\in J'}$ be given by Claim~\ref{Claim333} applied to $(\Phi(T_j))_{j\in J}$. Without loss of generality,  assume that $J'$ is countable and that  $(y,y_j)\in \supp(\Phi(T_j))$, for all $j\in J'$. Since $J'$ is countable,  Claim~\ref{Claim111} gives us that
\[\{(y,y_j)\mid j\in J'\}\subset \bigcup_{j\in J'}\supp(\Phi(T_j))\in\mathcal{F}.\]
Since the elements $(y_j)_{j\in J'} $ are pairwise distinct and $(Y,\mathcal{F})$ is uniformly locally finite, this gives us a contradiction.
\end{proof}

We can finally finish the proof of the lemma. For now on, we fix a bounded family $(T_j)_{j\in J}$ of rank $1$ operators in $\cstu[X]$ with mutually orthogonal images.  Define an equivalence relation $\sim$ on $J$ as follows. First, define a relation  (not necessarily an equivalence relation) $\sim'$ on $J$ by saying that $j\sim' j'$  either
\[
\supp(T_{j})\cap \supp(T_{j'})\neq \emptyset\ \ \text{or}\ \ \supp(\Phi(T_{j}))\cap \supp(\Phi(T_{j'}))\neq \emptyset.
\]
Then, say that $j\sim j'$ if  there exist $n\in\N$ and $j_1,\ldots,j_n\in J$, with $j_1=j$ and $j_n=j'$, such that $j_k\sim'j_{k+1}$, for all $k\in\{1,\ldots,k-1\}$. This defines a  partition on $J$, say $J=\bigsqcup_{i\in I}J_i$, for some index set $I$. By the definition of $\sim$, we have that 
  \[
  \supp(T_j)\cap \supp(T_{j'})= \emptyset\ \ \text{and}\ \ \supp(\Phi(T_j))\cap \supp(\Phi(T_{j'}))= \emptyset,
  \]
for all $j,j'\in J$ such that $j\not\sim j'$. 

By Lemma \ref{LemmaFiniteSupp}, Claim~\ref{Claim444} and Claim~\ref{Claim555},  $J_i$ is countable, for all $i\in I$. Therefore,  by Claim~\ref{Claim111} there exists an operator $T^{(i)}\in\cstu[X]$ such that $\supp(T^{(i)})\subset E$ and  
\[\supp(\Phi(T^{(i)}))=\bigcup_{j\in J_i}\supp(\Phi(T_j)).\]
By multiplying $T^{(i)}$ by an appropriate scalar if necessary,  assume that $|\langle T^{(i)}\delta_x,\delta_{x'}\rangle|\leq 1$, for all $x,x'\in X$. Since $\supp(T_j)\cap \supp(T_{j'})=\emptyset$, for all $j\not\sim j'$, it follows that $\supp(T^{(i)})\cap \supp(T^{(i')})=\emptyset$, for all $i\neq i'$. This shows that $\sum_{i\in I}T^{(i)}$ is a well-defined element of $\cstu[X]$ with support contained in $E$. 

By Lemma \ref{LemmaSpakulaWillett}, $\Phi: \cstu[X]\to \cstu[Y]$  is continuous in the strong operator topology. By Lemma \ref{LemmaSOTConv}, the series $\sum_{i\in I}T^{(i)}$ is convergent in the strong operator topology, so  
\[\Phi\Big(\sum_{i\in I}T^{(i)}\Big)=\sum_{i\in I}\Phi(T^{(i)}).\]
Since  \[\supp(\Phi(T_j))\cap \supp(\Phi(T_{j'}))= 0,\]
for all $j\not\sim j'$, it follows that  $\supp(\Phi(T^{(i)}))\cap \supp(\Phi(T^{(i')}))=\emptyset$, for all $i\neq i'$. Therefore, 
\[\supp\Big(\sum_{i\in I}\Phi(T^{(i)})\Big)=\bigcup_{i\in I} \supp\big(\Phi(T^{(i)})\big)=\bigcup_{j\in J}\supp\big(\Phi(T_j)\big).\]
As $\supp(\Phi(\sum_{i\in I}T^{(i)}))\in\mathcal{F}$, this completes the proof.
\end{proof}

We can now prove that a    isomorphism $\Phi:\cstu[X]\to \cstu[Y]$ between the algebraic uniform Roe algebras of uniformly locally finite coarse spaces must satisfy a `coarse-like' property (cf. Theorem~\ref{T.Approximation}). 

\begin{thm}\label{ThmAlgebraicMain}
Let $(X,\mathcal{E})$ and $(Y,\mathcal{F})$ be uniformly locally finite coarse spaces and let   $\Phi:\cstu[X]\to\cstu[Y]$ be an    isomorphism. For all  $E\in \mathcal{E}$, there exists $F\in \mathcal{F}$ such that, for all $T\in \cstu[X]$,
\[\supp(T)\subset E\ \ \text{implies}\ \ \supp(\Phi(T))\subset F.\]
\end{thm}

\begin{proof}
Suppose otherwise.  Then, there exists $E\in\mathcal{E}$ such that for all $F\in\mathcal{F}$ there exists $T^F\in \cstu[X]$ with $\supp(T^F)\subset E$ and $\supp(\Phi(T^F))\not\subset F$. Without loss of generality,  assume that $\Delta_X\subset E$. By multiplying $T^F$ by an appropriate scalar if necessary,  assume that $\|T^F\|\leq 1$, for all $F\in\mathcal{F}$. For each $F\in \mathcal{F}$, pick $(y^F_1,y^F_2)\in \supp(\Phi(T^F))$ such that $(y^F_1,y^F_2)\not\in F$. By Lemma \ref{LemmaSpakulaWillett}, $\Phi$ is continuous in the strong operator topology, therefore, using Lemma \ref{LemmaSOTConv}, it follows that \[\Phi(T^F)=\Phi\Big(\sum_{(x_1,x_2)\in E}T^F_{x_1x_2}\Big)=\sum_{(x_1,x_2)\in E}\Phi(T^F_{x_1x_2}).\]
So, for each $F\in\mathcal{F}$, there exists $(x^F_1,x^F_2)\in E$ such that $(y^F_1,y^F_2)\in\supp(\Phi(T^F_{x^F_1x^F_2}))$. 

We make  $\mathcal{F}$ into a directed set by setting $F_1\preceq F_2$ if $F_1\subset F_2$. By Lemma \ref{LemmaNetsCoarseStructure}, there exists a cofinal subset $I$ of $\mathcal{F}$, a subset $J$ of  $I$, and a map $\varphi:I\to J$ such that
\begin{enumerate}[(i)]
\item   $x^F_2\neq x^{F'}_2$, for all $F\neq F'$ in $J$, and
\item $x^F_1=x^{\varphi(F)}_1$ and $x^F_2=x^{\varphi(F)}_2$, for all $F\in I$.
\end{enumerate}

Notice that $\|T^F_{x^F_1x^F_2}\|\leq \|T^F\|\leq 1$. Hence, by Item (i),  $(T^F_{x^F_1x^F_2})_{F\in J}$ is a bounded family of rank $1$ operators in $\cstu[X]$ with mutually orthogonal images. Therefore, by Lemma \ref{LemmaAlgebraicMain}, it follows that 
\[F'\coloneqq\bigcup_{F\in J}\supp\Big(\Phi\big(T^F_{x^F_1x^F_2}\big)\Big)\in \mathcal{F}.\]
As $I$ is cofinal in $\mathcal{F}$, we can pick $F\in I$  such that $F'\subset F$. Fix such $F\in I$. By  our choice of $(y^F_1,y^F_2)$, we have that  $(y^{F}_1,y^{F}_2)\not\in F$. On the other hand, 
\[
(y^{F}_1,y^{F}_1)\in \supp\Big(\Phi\big(T^F_{x^F_1x^F_2}\big)\Big)=\supp\Big(\Phi\big(T^{\varphi(F)}_{x^{\varphi(F)}_1x_2^{\varphi(F)}}\big)\Big)\subset  F';
\]
 contradiction.
\end{proof}

We need to introduce some notation which will be used in the following
 lemmas. Let $(X,\mathcal{E})$ and $(X,\mathcal{F})$ be uniformly locally finite coarse spaces and let $U:\ell_2(X)\to\ell_2(Y)$ be a unitary isomorphism which spatially implements a    isomorphism between $\cstu[X]$ and $\cstu[Y]$. For each $x\in X$ and $y\in Y$,  define
\[X_y=\{x\in X\mid e_{xx}U^*e_{yy}U\neq 0\}\]
and 
\[Y_x=\{y\in Y\mid Ue_{xx}U^*e_{yy}\neq 0\}.\]

\begin{lemma}\label{LemmaDiamXY}
Let $(X,\mathcal{E})$ and $(X,\mathcal{F})$ be uniformly locally finite coarse spaces and let $U:\ell_2(X)\to\ell_2(Y)$ be a unitary isomorphism which spatially implements a    isomorphism between $\cstu[X]$ and $\cstu[Y]$. The following holds.

\begin{enumerate}[(i)]
\item There exists $E\in \mathcal{E}$ such that $(x,x')\in E$, for all $y\in Y$ and all $x,x'\in X_y$. In particular, $\sup_{y\in Y}|X_y|\leq \sup_{x\in X}|E_x|$.
\item There exists $F\in \mathcal{F}$ such that $(y,y')\in F$, for all $x\in X$ and all $y,y'\in Y_x$.  In particular, $\sup_{x\in X}|Y_x|\leq \sup_{y\in Y}|F_y|$.
\end{enumerate}
\end{lemma}

\begin{proof}
By symmetry, we only need to prove (i). Let $E$ be given by Theorem \ref{ThmAlgebraicMain} applied to the    isomorphism $T\in \cstu[Y]\mapsto U^*TU \in\cstu[X]$ and $\Delta_Y\in \mathcal{F}$. So, 
\[\supp(U^*e_{yy}U)\subset E,\]
for all $y\in Y$. Without loss of generality,  assume that $E$ is symmetric. Let $y\in Y$ and $x,x'\in X_y$. Since $e_{xx}U^*e_{yy}U\neq 0$, there exists $x''\in X$ such that $e_{xx}U^*e_{yy}U\delta_{x''}\neq 0$. As $U^*e_{yy}U$ is a rank $1$ operator and  $e_{x'x'}U^*e_{yy}U\neq 0$, we have that $e_{x'x'}U^*e_{yy}U\delta_{x''}\neq 0$. In other words, this gives us that 
\[\langle U^*e_{yy}U\delta_{x},\delta_{x''}\rangle\neq 0\ \ \text{and}\ \ \langle U^*e_{yy}U\delta_{x'},\delta_{x''}\rangle\neq 0.\]
So, $(x,x''),(x',x'')\in\supp(U^*e_{yy}U)\subset E$. As $E$ is symmetric, this implies that $(x,x')\in E\circ E$. The entourage $E\circ E$ has the desired property.
\end{proof}

\begin{lemma}\label{LemmaCardXB}
Let $(X,\mathcal{E})$ and $(Y,\mathcal{F})$ be uniformly locally finite coarse spaces and let $U:\ell_2(X)\to\ell_2(Y)$ be a unitary isomorphism which spatially implements a    isomorphism between $\cstu[X]$ and $\cstu[Y]$. The following holds.
\begin{enumerate}[(i)]
\item $|\bigcup_{y\in B}X_y|\geq |B|$, for all $B\subset Y$.
\item $|\bigcup_{y\in B}Y_x|\geq |B|$, for all $B\subset X$.
\end{enumerate}
\end{lemma}

\begin{proof}
By symmetry, it suffices to prove (i). Fix $B\subset Y$ and let $X_B=\bigcup_{y\in B}X_y$. Since
 $\supp(U^*e_{yy}U)\subset X_y\times X_y$, for all $y\in Y$, it follows  that 
 \[
 \bigcup_{y\in B}\supp(U^*e_{yy}U)\subset X_B\times X_B.
 \]
  Therefore, for each $y\in B$, we can naturally identify $U^*e_{yy}U$ with an operator in $\ell_2(X_B)$. For each $y\in B$, let $x_y\in X_B$ be such that $U^*e_{yy}U\delta_{x_y}\neq 0$. Since the images of $U^*e_{yy}U$ and $U^*e_{y'y'}U $ are orthogonal for all $y\neq y'$ in $B$, we have that 
$(U^*e_{yy}U\delta_{x_y})_{y\in B}$ is an orthogonal family of nonzero vectors in $\ell_2(X_B)$. So $|B|\leq |X_B|$. 
\end{proof}

\begin{lemma}\label{LemmaInjection}
Let $(X,\mathcal{E})$ and $(X,\mathcal{F})$ be uniformly locally finite coarse spaces and let $U:\ell_2(X)\to\ell_2(Y)$ be a unitary isomorphism which spatially implements a    isomorphism between $\cstu[X]$ and $\cstu[Y]$. There exist injections $f:X\to Y$ and $g:Y\to X$ such that 
\begin{enumerate}[(i)]
\item $f(x)\in Y_x$, for all $x\in X$, and
\item $g(y)\in X_y$, for all $y\in Y$.
\end{enumerate}
\end{lemma}

\begin{proof}
By symmetry, it suffices to prove (i). Define $\varphi:X\to \mathcal{P}(Y)$ by letting $\varphi(x)=Y_x$, for all $x\in X$. By Lemma \ref{LemmaCardXB}, we know that 
\[|B|\leq \Big|\bigcup_{x\in B}Y_x\Big|=\Big|\bigcup_{x\in B}\varphi(x)\Big|,\]
for all $B\subset X$. By Lemma \ref{LemmaDiamXY}(ii), we also have that $Y_x$ is finite, for all $x\in X$. Therefore, by Hall's marriage theorem (see \cite{HallBook}, Theorem 5.1.2), there exists a map $\psi:\{Y_x\mid x\in X\}\to Y$ such that $\psi(Y_x)\in Y_x$, for all $x\in X$, and $\psi(Y_x)\neq \psi(Y_{x'})$, for all $x\neq x'$. Therefore, the map $f=\psi\circ \varphi$ is an injection such that $f(x)\in Y_x$, for all $x\in X$.
\end{proof}

\begin{thm}\label{ThmV}
Let $(X,\mathcal{E})$ and $(Y,\mathcal{F})$ be uniformly locally finite coarse spaces and suppose that $\Phi:\cstu[X]\to \cstu[Y]$ is a    isomorphism. Then, there exists a unitary operator $V\in \cstu[Y]$ such that $V\Phi(\ell_\infty(X))V^*=\ell_\infty$.
\end{thm}

 \begin{proof}
 Let $f:X\to Y$ and $g:Y\to X$ be the injections given by Lemma \ref{LemmaInjection}. Using K\"{o}nig's proof of the Cantor-Scr\"{o}der-Bernstein theorem to the injections $f$ and $g$, we obtain a bijection $h:Y\to X$ such that, for all $y\in Y$, we  have that either $h(y)=g(y)$ or  $y\in \text{Im}(f)$  and $h(y)=f^{-1}(y)$.
 
Let $F\in\mathcal{F}$  be given by Lemma \ref{LemmaDiamXY}(ii). Without loss of generality,  assume that $F$ is symmetric. So, $(y,y')\in F$, for all $x\in X$ and all $y,y'\in Y_x$. 

\begin{claim}\label{Claim1x} $Y_{h(y)}\subset F_y$, for all $y\in Y$.
\end{claim}

\begin{proof} 
Fix $y\in Y$. Suppose $y\in \text{Im}(f)$ and $h(y)=f^{-1}(y)$. By the definition of~$f$, we have that  $y=f(h(y))\in Y_{h(y)}$. Therefore, $(y,y')\in F$, for all $y'\in Y_{h(y)}$. So, $Y_{h(y)}\subset F_y$. Suppose now that $h(y)=g(y)$. By the definition of $g$, $g(y)\in X_y$. Hence, by the definition of $X_y$, $e_{g(y)g(y)}U^*e_{yy}U\neq 0$ and we must have that $Ue_{g(y)g(y)}U^*e_{yy}\neq 0$. Therefore, by the definition of $Y_{h(y)}$, it follows that $y\in Y_{h(y)}$ and we conclude that $Y_{h(y)}\subset F_y$. This finishes the proof of the claim.
\end{proof} 

For each $x\in X$, let $\xi_x=U\delta_x$. So $(\xi_x)_{y\in X}$ is an orthonormal basis of $\ell_2(Y)$. We define a unitary operator $V\in B(\ell_2(Y))$ by letting
\[V\xi=\sum_{x\in X}\langle\xi,\xi_x\rangle \delta_{h^{-1}(x)},\]
for all $\xi\in \ell_2(Y)$. 

\begin{claim}\label{Claim2x} $\supp(V) \subset F$. In particular, $V\in \cstu[Y]$.
\end{claim} 

\begin{proof} 
Let $y,y'\in Y$. Then, by the definition for $V$, we have that
\begin{align*}
\big|\langle V\delta_y,\delta_{y'}\rangle\big|&=\Big|\Big\langle\sum_{x\in X}\langle\delta_y,\xi_x\rangle \delta_{h^{-1}(x)},\delta_{y'}\Big\rangle\Big|\\
& =|\langle\delta_y, \xi_{h(y')}\rangle|\\
&=|\langle \delta_y,U\delta_{h(y')}\rangle|\\
&=|\langle U^*\delta_y,\delta_{h(y')}\rangle|\\
&=\|e_{h(y')h(y')}U^*e_{yy}\|\\
&= \|Ue_{h(y')h(y')}U^*e_{yy}\|
\end{align*}
By Claim~\ref{Claim1x},  $Y_{h(y')}\subset F_{y'}$. Hence,  $(y',y)\not\in F$ implies $y\not\in Y_{h(y')}$. By the definition of $Y_{h(y')}$, $y\not\in Y_{h(y')}$ implies $Ue_{h(y')h(y')}U^*e_{yy}=0$. Since $F$ is symmetric, we conclude that $(y,y')\not\in F$ implies $\langle V\delta_y,\delta_{y'}\rangle=0$. Therefore, $\supp(V)\subset F$ and the claim is proven.
\end{proof} 

In order to finish the proof,  notice that  $VU\ell_\infty(X)U^*V^*=\ell_\infty (X)$. Indeed, a simple computation gives us   that
\[V^*\xi=\sum_{x\in X}\langle \xi,\delta_{h^{-1}(x)}\rangle\xi_x,\]
for all $\xi\in \ell_2(Y)$. Hence, for any $x\in X$, we have that 
\[VUe_{xx}U^*V^*\xi=\langle \xi,\delta_{h^{-1}(x)}\rangle\delta_{h^{-1}(x)}=e_{h^{-1}(x)h^{-1}(x)}\xi,\]
for all $\xi\in \ell_2(Y)$. So, $VUe_{xx}U^*V^*=e_{h^{-1}(x)h^{-1}(x)}$. Since $h$ is a bijection, we conclude that  $VU\ell_\infty(X)U^*V^*=\ell_\infty (X)$.  This finishes the proof.
 \end{proof}

\begin{proof}[Proof of (ii)$\Rightarrow$(iii) of Theorem \ref{T.ThmMAIN}]
Let $\Phi:\cstu[X]\to \cstu[Y]$ be a    isomorphism and let $V$ be given by Theorem \ref{ThmV}. Define a    isomorphism $\Psi:\cstu[X]\to \cstu[Y]$ by letting  $\Psi(T)=V\Phi(T)V^*$, for all $T\in \cstu[X]$. Clearly, $\Psi(\ell_\infty(X))\subset \ell_\infty(Y)$ and $\Psi$ extends to a    isomorphism $\cstu(X)\to \cstu(Y)$, so the proof is complete. 
\end{proof}

We now turn to the proof of (iii) $\Rightarrow$ (i) of Theorem \ref{T.ThmMAIN}. We show that a    isomorphism  $\Phi:\cstu(X)\to \cstu(Y)$ sending $\ell_\infty(X)$ to $\ell_\infty(Y)$ must satisfy the same `coarse-like' property of Theorem \ref{ThmAlgebraicMain} above (see Theorem \ref{ThmMainSpecialCase} below). In particular, we obtain that  $\Phi(\cstu[X])\subset \cstu[Y]$.

\begin{lemma}\label{LemmaBij}
Let $(X,\mathcal{E})$ and $(Y,\mathcal{F})$ be uniformly locally finite coarse spaces and consider a    isomorphism $\Phi:\cstu(X)\to \cstu(Y)$ such that $\Phi(\ell_\infty(X))\subset\ell_\infty(Y)$. There exists a bijection $f:X\to Y$ such that  

\begin{enumerate}[(i)]
\item $\Phi(e_{xx})=e_{f(x)f(x)}$, for all $x\in X$, and 
\item for all $x,x'\in X$ there exists $\lambda_{x'x}\in\C$, with $|\lambda_{x'x}|=1$, such that $\Phi(e_{x'x})=\lambda_{x'x} e_{f(x')f(x)}$.
\end{enumerate}
In particular $\Phi^{-1}(\ell_\infty(Y))\subset\ell_\infty(X)$.
\end{lemma}

\begin{proof}
(i)  Fix $x\in X$.  As $e_{xx}\in \ell_\infty(X)$, we have that $\Phi(e_{xx})\in \ell_\infty(Y)$. Since $e_{xx}$ has rank $1$, so does $\Phi(e_{xx})$ . Therefore, there exists $y_x\in Y$ and $\lambda_x\in \C\setminus \{0\}$ such that $\Phi(e_{xx})=\lambda_x e_{y_xy_x}$. Since $\sigma(e_{xx})=\{0,1\}$, it follows that $\sigma(\Phi(e_{xx}))=\{0,1\}$ and we must have $\lambda_x=1$. 
 
Define $f(x)=y_x$, for all $x\in X$, and let us show that $f$ is a bijection. Say $x\neq x'$. Then $e_{xx}e_{x'x'}=0$, so \[e_{f(x)f(x)}e_{f(x')f(x')}=\Phi(e_{xx}e_{x'x'})=0.\] Hence, $f(x)\neq f(x')$ and  $f$ is injective. Say $f$ is not surjective. Then, there exists $y\in Y$ such that $\Phi(e_{xx})\neq e_{yy}$, for all $x\in X$. Then, $e_{yy}e_{f(x)f(x)}=0$, for all $x\in X$. This implies that  $\Phi^{-1}(e_{yy})e_{xx}=0$, for all $x\in X$, so  $\Phi^{-1}(e_{yy})=0$. Since $\Phi^{-1}$ is a    isomorphism, this gives us a contradiction. This shows that  $f$ is a bijection. 

(ii) Let $f:X\to Y$ be the bijection in Item (i). Fix $x,x'\in X$. 
Using Item (i) and that $e_{x'x}^*=e_{ x x'}$,  $e_{x'x}=e_{x'x}e_{xx}$ and $e_{xx'}=e_{xx'}e_{x'x'}$,  we have that
\begin{align*}
\langle \Phi(e_{x'x})\delta_y,\delta_{y'} \rangle &=\langle \Phi(e_{x'x})e_{f(x)f(x)}\delta_y,\delta_{y'} \rangle\\
&=\langle e_{f(x)f(x)}\delta_y,\Phi(e_{xx'})\delta_{y'} \rangle\\
&=\langle e_{f(x)f(x)}\delta_y,\Phi(e_{xx'})e_{f(x')f(x')}\delta_{y'} \rangle.
\end{align*}
Since $\Phi(e_{x'x})\neq 0$, it follows that $\langle \Phi(e_{x'x})\delta_y,\delta_{y'} \rangle \neq 0$ if and only if  $y=f(x)$ and $y'=f(x')$. So, $\Phi(e_{x'x})=\lambda_{x'x} e_{f(x')f(x)}$, for some $\lambda_{x'x}\in \C$. At last, notice that $|\lambda_{x'x}|=\|\Phi(e_{x'x})\|=1$. 
\end{proof}

\begin{lemma}\label{LemmaBijU}
Let $(X,\mathcal{E})$ and $(Y,\mathcal{F})$ be uniformly locally finite coarse spaces and consider a    isomorphism $\Phi:\cstu(X)\to \cstu(Y)$ such that $\Phi(\ell_\infty(X))\subset\ell_\infty(Y)$. Let $U:\ell_2(X)\to\ell_2(Y)$ be a unitary operator such that $\Phi(T)=UTU^*$, for all $T\in \cstu(X)$. Let $f:X\to Y$ be the bijection given by Lemma \ref{LemmaBij}. Then, for all $x\in X$, there exists $\lambda\in \C$, with $|\lambda|=1$, such that $U(\delta_x)=\lambda\delta_{f(x)}$.
\end{lemma}

\begin{proof}
Fix $x\in X$. By Lemma \ref{LemmaBij}, $Ue_{xx}U^*=e_{f(x)f(x)}$, so  $U\delta_x=e_{f(x)f(x)}U\delta_x$, for all $x\in X$. This gives us that $U\delta_x$ is a multiple of $\delta_{f(x)}$. Since $U$ is an isometry, it follows that $U(\delta_x)=\lambda\delta_{f(x)}$, for some $\lambda\in \C$, with $|\lambda|=1$.
\end{proof}

\begin{thm}\label{ThmMainSpecialCase}
Let $(X,\mathcal{E})$ and $(Y,\mathcal{F})$ be uniformly locally finite coarse spaces and let   $\Phi:\cstu(X)\to \cstu(Y)$ be a    isomorphism such that $\Phi(\ell_\infty(X))\subset\ell_\infty(Y)$. For all  $E\in \mathcal{E}$, there exists $F\in \mathcal{F}$ such that, for all $T\in \cstu(X)$,
\[\supp(T)\subset E\ \ \text{implies}\ \ \supp(\Phi(T))\subset F.\]
In particular $\Phi(\cstu[X])\subset \cstu[Y]$.
\end{thm}

\begin{proof}
Assume the conclusion of the theorem does not hold. Then, there exists $E\in\mathcal{E}$ such that for all symmetric entourage $F\in\mathcal{F}$ there exists $T^F\in \cstu(X)$ with $\supp(T^F)\subset E$ and $\supp(\Phi(T^F))\not\subset F$. We can assume that $\|T^F\|\leq 1$, for all $F\in\mathcal{F}$. For  $F\in \mathcal{F}$, pick $(y^F_2,y^F_1)\in \supp(\Phi(T^F))$ such that $(y^F_2,y^F_1)\not\in F$.  By continuity of $\Phi$ in the strong operator topology, Lemma \ref{LemmaSOTConv} and Lemma \ref{LemmaSpakulaWillett} imply  that \[\Phi(T^F)=\Phi\Big(\sum_{(x_1,x_2)\in E}T^F_{x_1x_2}\Big)=\sum_{(x_1,x_2)\in E}\Phi(T^F_{x_1x_2}).\]
For each $F\in\mathcal{F}$, pick $(x^F_1,x^F_2)\in E$ such that $(y^F_2,y^F_1)\in\supp(\Phi(T^F_{x^F_1x^F_2}))$.  Let $f:X\to Y$ be the bijection given in Lemma \ref{LemmaBij}. By Lemma \ref{LemmaBij},  $\supp(\Phi(T^F_{x^F_1x^F_2}))=\{(f(x^F_2),f(x^F_1))\}$, $y^F_1=f(x^F_1)$ and $y^F_2=f(x^F_2)$.

We make  $\mathcal{F}$ into a directed set by setting $F_1\preceq F_2$ if $F_1\subset F_2$. By Lemma \ref{LemmaNetsCoarseStructure}, we can pick a cofinal subset $I$ of $\mathcal{F}$, a subset $J$ of $ \mathcal{F}$ and a map $\varphi :I\to J$ such that
\begin{enumerate}[(i)]
\item $x_1^F\neq x^{F'}_1$ and $x_2^F\neq x^{F'}_2$, for all distinct $F,F'\in J$, and
\item $x^F_1=x^{\varphi(F)}_1$ and $x^F_2=x^{\varphi(F)}_2$, for all $F\in I$.
\end{enumerate} 

By Item (i) and Lemma \ref{LemmaSOTConv}, the sum \[\sum_{F\in J} e_{x^F_1x^F_2}\]
converges in the strong operator topology to an operator in  $\cstu[X]$. Let $(\lambda_{x^F_1x^F_2})_{F\in J}$ be given by Lemma \ref{LemmaBij}(ii). Then, as $\Phi$ is continuous in the strong operator topology, the sum
\[\sum_{F\in J}\Phi( e_{x^F_1x^F_2})=\sum_{F\in J} \lambda_{x^F_1x^F_2}e_{f(x^F_1)f(x^F_2)}\]
converges strongly to an operator $S$ in $\cstu(Y)$. Pick $S'\in \cstu[Y]$ such that $\|S-S'\|<1$. In particular, $\supp(S')\in\mathcal{F}$.

\begin{claim}\label{Claim:} $(f(x^F_2),f(x^F_1))\in \supp(S') $, for all $F\in J$.
\end{claim} 

\begin{proof} 
Notice that 
\[|\langle S\delta_y,\delta_{y'} \rangle|=\left\{\begin{array}{l l}
1, & \ \ \text{ if }(y,y')=(f(x^F_2),f(x^F_1))\ \ \text{ for some }\ \ F\in J,\\
0, &\ \ \text{ otherwise.}
\end{array}\right.\]
Let $F\in \mathcal{F}$. Since $\|S(\delta_{f(x^F_2)})-S'(\delta_{f(x^F_2)})\|<1$, we have that \[|\langle S(\delta_{f(x^F_2)})-S'(\delta_{f(x^F_2)}),\delta_{f(x^F_1)}\rangle|< 1.\]
This gives us that  $(f(x^F_2),f(x^F_1))\in \supp(S') $, and the claim is proved. 
\end{proof} 

Since $I$ is cofinal in $\mathcal{F}$, we can pick $F\in I$ such that $\supp(S')\subset F$. Fix such $F\in I$. By hypothesis, \[(f(x^F_2),f(x^F_1))=(y^F_2,y^F_1)\not\in F.\] Therefore, by Item (ii), we must have \[\Big(f(x^{\varphi(F)}_2),f(x^{\varphi(F)}_1)\Big)\not\in F.\]
Since $\varphi(F)\in J$ and $\supp(S')\subset F$, the claim 
above gives us a contradiction.
\end{proof}

\begin{proof}[Proof of (iii)$\Rightarrow$(i) of Theorem \ref{T.ThmMAIN}]
Let $f:X\to Y$ be the bijection given in Lemma \ref{LemmaBij}. By Lemma \ref{LemmaSpakulaWillett}, there exists a unitary isomorphism $U:\ell_2(X)\to \ell_2(Y)$ such that $\Phi(T)=UTU^*$, for all $T\in \cstu[X]$. By Lemma \ref{LemmaBijU}, we have that $\langle U\delta_x,\delta_{f(x)}\rangle\neq 0$, for all $x\in X$. Let $g=f^{-1}$.\\

\begin{claim}\label{Claim1y} $f$ and $g$ are coarse maps.
\end{claim} 

\begin{proof} 
By symmetry, it suffices to show that $f$ is coarse. Let $E\in\mathcal{E}$ and let $F\in\mathcal{F}$ be given by Theorem \ref{ThmMainSpecialCase}. Without loss of generality, we can assume that $F$ is symmetric. For all $x,x'\in X$, we have that 
\begin{align}\label{equality}
\langle \Phi(e_{xx'})\delta_{f(x')},\delta_{f(x)}\rangle &= \langle Ue_{xx'}U^*\delta_{f(x')},\delta_{f(x)}\rangle\\
&=\langle e_{xx'}U^*\delta_{f(x')},U^*\delta_{f(x)}\rangle\nonumber\\
&=\langle \langle \delta_{x'},U^*\delta_{f(x')}\rangle \delta_{x},U^*\delta_{f(x)}\rangle\nonumber\\
&=\langle \delta_{x},U^*\delta_{f(x)}\rangle\langle \delta_{x'},U^*\delta_{f(x')}\rangle\nonumber\\
&=\langle U\delta_{x},\delta_{f(x)}\rangle\langle U\delta_{x'},\delta_{f(x')}\rangle\nonumber
\end{align}
Therefore, by the definition of $f$, $\langle \Phi(e_{xx'})\delta_{f(x')},\delta_{f(x)}\rangle\neq 0$, for all $x,x'\in X$. Hence, if $(x,x')\in E$, this gives us that  \[(f(x'),f(x))\in \supp(\Phi(e_{xx'}))\subset F.\]
Since $F$ is symmetric, we are done.
\end{proof} 

\begin{claim}\label{Claim2y} $g\circ f$ and $f\circ g$ are close to $\Id_X$ and $\Id_Y$, respectively.
\end{claim}

\begin{proof} By symmetry, it suffices to prove that $f\circ g$ is close to $\Id_Y$. Let $E\in\mathcal{E}$ be given by Theorem \ref{ThmMainSpecialCase} applied to the    isomorphism $\Phi^{-1}:\cstu(Y)\to \cstu(X)$ and the diagonal $\Delta_Y\subset Y\times Y$.  By \ref{equality} above, we have that 
\[\langle \Phi(e_{g(y)g(y)})\delta_{y},\delta_{f(g(y))}\rangle =\langle U\delta_{g(y)},\delta_{f(g(y))}\rangle\langle U\delta_{g(y)},\delta_{y}\rangle,\]
for all $y\in Y$. Therefore, by the definition of $f$ and $g$, we get that
\[\langle \Phi(e_{g(y)g(y)})\delta_{y},\delta_{f(g(y))}\rangle =\langle U\delta_{g(y)},\delta_{f(g(y))}\rangle\langle \delta_{g(y)},U^*\delta_{y}\rangle\neq 0,\]
for all $y\in Y$. We conclude that  \[(y,f(g(y))\in  \supp(\Phi(e_{g(y)g(y)}))\subset E,\]
for all $y\in Y$. 
\end{proof}
This finishes the proof.
\end{proof}

\section{Appendix: Generically absolute isomorphisms} 
\label{S.Forcing} 
\setcounter{claim}{0}
In this appendix  some familiarity with models of set theory and absoluteness is desirable; see
for example  \cite[Section~I.16 and Section~II.5]{Ku:Set} or \cite{Jech:SetTheory}. 
Before  defining the  notion of a generically absolute isomorphism 
between uniform Roe algebras, we should point out that the absoluteness theorem for $\Pi^1_1$ statements (\cite{Jech:SetTheory}, Theorem  25.4) implies that 
every isomorphism between uniform Roe algebras associated to metric spaces is generically absolute. 
Therefore Theorem~\ref{T.ThmRigIsoImpliesCoarseEqui.absolute}
 below is a generalization of 
the instance of Theorem~\ref{T.ThmRigIsoImpliesCoarseEqui} for metric spaces. 

Suppose that $M\subseteq N$ are two transitive models of a large enough fragment of ZFC with the same set of ordinals. (Because of metamathematical 
considerations not directly relevant to our discussions, we cannot assume that ZFC is consistent and therefore have to work with 
a model of a large enough finite fragment of ZFC; see \cite[II]{Ku:Set} for an extensive discussion.)
Furthermore suppose that $(X,\cE)$ and $(Y,\cF)$ are coarse spaces,  $\Phi\colon \cstu(X)\to \cstu(Y)$ is an    isomorphism and all those objects are in $M$. 
 
Since $M\subseteq N$, all of these objects belong to $N$. However, they need not be objects of the required form. 
For example,  in $N$ the set $\cE$ is a collection of subsets of $X\times X$, and it satisfies (i), (iii), (iv), and (v) of Definition~\ref{DefiCoarseSpace}. 
(This is a consequence of the absoluteness of the notions involved in these axioms; see  \cite[II.4]{Ku:Set}.)
However, $N$ may contain a subset of $X$ that does not belong to $M$, and in this case (ii) of  Definition~\ref{DefiCoarseSpace} will fail 
for $\cE$. The way to remedy this issue is to take the coarse structure on $X$ generated by $\cE$ in the model $N$.

We take (writing $\cP(E)=\{E'\mid  E'\subseteq E\}$)
\[
\cE^N=\bigcup_{E\in \cE} \cP(E)
\]
 \emph{as computed in $N$}. Then we have that $(X, \cE^N)$ is a coarse space in $N$.\footnote{Purists may object our not distinguishing  $X^M$ from $X^N$ and using $X$
 to denote both sets; this is however the same set and we find writing $\ell_2(X)^N$  preferable to writing $\ell_2(X^N)^N$. The notation 
 $K(\ell_2(X)^N)^N$ appears to be a necessary evil.}

We proceed to define  interpretations of  other relevant objects in the model $N$. 
For   $\ell_2(X)^N$ we take the completion of  $\ell_2(X)^M$; this space has $\delta_x$, for $x\in X$, 
as an orthonormal basis and it clearly agrees with $\ell_2(X)$ as computed in $N$. 
The coarse space $(Y,\cF^N)$ and the Hilbert space $\ell_2(Y)^N$ are defined analogously. 
Then $U^M$ is a linear isometry between dense subspaces of 
$\ell_2(X)^N$ and $\ell_2(Y)^N$, and we let $U^N$ denote its continuous extension. 
This is a unitary. 

It remains to see how the uniform Roe algebras $\cstu(X)^N$ and $\cstu(Y)^N$ 
relate to the uniform Roe algebras $\cstu(X)^M$ and $\cstu(Y)^M$. 
A minor inconvenience is caused by the following two facts.  (We assume the `worst case scenario'' that $N$ contains a subset of $X$ which is not in $M$.) 
\begin{enumerate}
\item   The algebra $\ell_\infty(X)^N$ is equal the closure of $\ell_\infty(X)^M$ in the weak operator topology, 
and it is strictly lager than the closure of $\ell_\infty(X)^M$ in the norm topology. 
\item   The algebra $K(\ell_2(X)^N)^N$ is equal to the closure of $K(\ell_2(X)^M)^M$ in the norm topology. 
It is strictly \emph{smaller} than the closure of $K(\ell_2(X)^M)^M$ in the weak operator  topology. 
\end{enumerate}

Lemma~\ref{L.RoeCode} will provide us with  a recipe for how to compute $\cstu(X)^N$ directly from $(X,\cE^M)$. 
(The reason for the absence of the  superscripts ${}^M$ and ${}^N$ in the statement of Lemma~\ref{L.RoeCode} is that 
this lemma has nothing to do with models of  fragments of ZFC.)
The proof of the lemma is omitted, being an immediate consequence of the definition of $\cstu(X)$. 
  
\begin{lemma} \label{L.RoeCode} Suppose that $(X,\cE)$ is a uniformly locally discrete coarse space. 
For $E\in \cE$ the set 
\[
\cZ_E=\{T\in B(\ell_2(X)): \supp(T)\subseteq E\}. 
\]
is closed in the weak operator topology and contained in $\cstu(X)$. 
If $\cE_1\subseteq \cE$ is cofinal with respect to inclusion, then  
 $\cstu(X)=\overline{\bigcup_{E\in \cZ_1} \cZ_E}^{\|\cdot\|}$. \qed
 \end{lemma} 
 
Back to interpreting notions in models $M$ and $N$, we note that 
for $E\in \cE^M$, the set $(\cZ_E)^N$ is equal 
to the closure of $(\cZ_E)^M$ in the weak operator topology. 
We can now define
\[
\cstu(X)^N=\overline{\bigcup_{E\in \cE^M} (\cZ_E)^N}^{\|\cdot\|}
\]
where the norm closure is computed in $N$. Since $\cE^M$ is cofinal in $\cE^N$,  Lemma~\ref{L.RoeCode} implies that 
$\cstu(X)^N$ defined in this way  coincides with the uniform Roe algebra of the coarse space $(X,\cE^N)$ as computed in $N$.

 Suppose that $\Phi\colon \cstu(X)\to \cstu(Y)$ is a    isomorphism. 
Then 
$X,Y,\cstu(X), \cstu(Y)$,  and $\Phi$ all belong to a large enough rank initial segment 
$R(\theta)$ (commonly denoted $V_\theta$; we use Kunen's notation) 
 of von Neumann's cumulative universe for set theory.  
Let $M_0$ be a countable elementary submodel of $R(\theta)$ 
containing 
$X,Y,\cstu(X), \cstu(Y), \Phi$, and the unitary $U$ implementing $\Phi$, 
 and let $M$ denote the Mostowski collapse of $M_0$. 
This is a countable transitive model isomorphic to $M_0$, 
and it contains copies 
$X^M,Y^M,\cstu(X)^M, \cstu(Y)^M, \Phi^M$, and $U^M$ 
of the above objects. (This time we write $X^M$ because $X$ may not belong to $M$.)
By elementarity, $M$ will satisfy the assertion 
\[
(\forall T\in B(\ell_2(X))) T\in \cstu(X)\Leftrightarrow UTU^*\in \cstu(Y). 
\]
We proceed to describe how an extension $N$ of a model $M$ that will serve our purpose is obtained. 
 Suppose that   $J$ is a set in $M$.  Observe that $\D^M$  is a countable dense subset of $\D$
 and that $(\D^J)^M$ is a countable dense subset of $\D^J$.  Therefore if $A\in M$ and  
 $A\subseteq (\D^J)^M$, then $A$ is nowhere dense in $\D^J$ if and only if the assertion 
 `$A$ is nowhere dense in $\D^J$' holds in  $M$. 

By the Baire category theorem, we can choose $G\in \D^J$ such that $G$ does not belong to the closure
 of any nowhere dense subset of $\D^J$ that belongs to $M$. 
 Such   $G$ is said to be  \emph{generic over $M$}. Then $G$ is generic (in the technical sense from the theory of forcing) for the poset 
 of all nonempty open subsets of $\D^J$ ordered by the inclusion.  
  A transitive model $M[G]$ that contains $G$ and includes $M$  can be formed
 as in \cite[IV.2]{Ku:Set}.

\begin{defi} \label{Def.Generically}. 
An isomorphism $\Phi\colon \cstu(X)\to \cstu(Y)$ 
implemented by a unitary $U$
 is  generically absolute
  if and only if  for all $M$ and $M[G]$ as in the previous paragraph, 
 $U$  implements an isomorphism between $\cstu(X)$ and $\cstu(Y)$.
\end{defi} 

A proof of the following lemma is now straightforward. 

\begin{lemma} 
A    isomorphism $\Phi\colon \cstu(X)\to \cstu(Y)$ 
between uniform Roe algebras of 
coarse spaces is generically absolute  if and only if  it  satisfies the conclusion of Lemma~\ref{LemmaCoarseSpakulaWillett}. 
More precisely, there is a function $f\colon \cE\times \N\to \cF$ such that 
for all $E\in \cE$,  every $n\geq 1$, and every  $T\in \cstu(X)$ such that $\supp(T)\subseteq E$
there exist $S\in \cstu(Y)$ such that $\supp(S)\subseteq f(E,n)$ and 
$\|\Phi(T)-S\|<1/n$. \qed
\end{lemma} 

The proof of Theorem~\ref{T.ThmRigIsoImpliesCoarseEqui} 
in which the smallness assumption is replaced by the assumption that $\Phi$ be generically absolute gives the following.

\begin{theorem} \label{T.ThmRigIsoImpliesCoarseEqui.absolute}
 Suppose that $(X,\cE)$ and $(Y,\cF)$ are  uniformly locally finite coarse spaces. 
If  $ \cstu(X)$ and   $\cstu(Y)$ are rigidly    isomorphic via a generically absolute isomorphism,  
then $X$ and $Y$ are coarsely equivalent. \qed
\end{theorem} 

We do not know whether it is possible to construct an isomorphism between uniform Roe algebras that is not generically  absolute.

\begin{acknowledgments}
The authors  would like to thank R. Willett for  suggesting the problem of rigidity of  uniform Roe algebras for not necessarily metrizable coarse spaces and   useful conversations. The authors would also like to thank the Fields Institute 
for allowing them to make use of its facilities. 
Some final touches were added after the very stimulating 
``Approximation Properties in Operator Algebras and Ergodic Theory''
workshop at UCLA/IPAM. IF would like to thank 
Hanfeng Li and Wouter van Limbeek for stimulating conversations and 
the organizers of the workshop for the invitation. 
\end{acknowledgments}

\end{document}